\documentclass[11pt]{amsart}
\usepackage{amsmath, amssymb}
\usepackage{fancyhdr}
\usepackage{enumitem}
\usepackage{geometry} 
\usepackage[colorlinks]{hyperref}
\newtheorem{thm}{Theorem}[section]
\newtheorem{thmx}{Theorem}

\newtheorem{prop}{Proposition}[section]
\newtheorem{rem}{Remark}[section]
\newtheorem{lem}{Lemma}[section]
\newtheorem{ex}{Example}[section]
\newtheorem{defn}{Definition}[section]
\newtheorem{cor}{Corollary}[section]

\begin{document}

\pagestyle{fancy}
\fancyhead{} 
\fancyhead[EC]{Helleland}   
\fancyhead[EL,OR]{\thepage}   
\fancyhead[OC]{Wick-rotations of \\ pseudo-Riemannian Lie groups}   
\fancyfoot{} 

\title{\large{\textbf{Wick-rotations of \\ pseudo-Riemannian Lie groups}}}
\author{\textbf{Christer Helleland}}
\date{\today}

\maketitle

\begin{center} \vspace{0.3cm}      
Faculty of Science and Technology,\\     
 University of Stavanger,\\  N-4036 Stavanger, Norway         
\vspace{0.3cm} \end{center}

\begin{center} \texttt{u17ch0@gmail.com}
\end{center}

\begin{abstract} We study Wick-rotations of left-invariant metrics on Lie groups, using results from real GIT (\cite{W2}, \cite{W3}). An invariant for Wick-rotation of Lie groups is given, and we describe when a pseudo-Riemannian Lie group (a Lie group with a left-invariant metric) can be Wick-rotated to a Riemannian Lie group. We define a Cartan involution of a general Lie algebra, and prove a general version of $\acute{E}$. Cartan's result, namely the existence and conjugacy of Cartan involutions.  \end{abstract}

\section{Introduction}
This paper is motivated first of all by the study of Wick-rotations of pseudo-Riemannian manifolds defined in \cite{W2}. Given a pseudo-Riemannian manifold $(M,g)$ of signature $(p,q)$, then it is interesting know whether it can be Wick-rotated to another space $(\tilde{M},\tilde{g})$ (w.r.t a fixed point $p\in M\cap \tilde{M}$) of signature $\tilde{p}+\tilde{q}=p+q$. In (\cite{W1}, \cite{W2}, \cite{W3}) the isometry action of the pseudo-orthogonal group $O(p,q)$ acting on tensors restricted to $p$ is explored. For instance it is proved that if $\tilde{p}=0$ (i.e $\tilde{g}$ is Riemannian) then there is a Cartan involution of the metric $\theta\in O(p,q)$ (at $p$) which fixes the Riemann tensor $R$ under the isometry action, i.e $\theta\cdot R=R$. Thus $(M,g)$ is \emph{Riemann purely electric} $(RPE)$ at $p$. More generally it is proved that for a space to be \emph{purely electric} (respectively \emph{purely magnetic}) or (RPE) (respectively \emph{Riemann purely magnetic}) is preserved under a Wick-rotation at a common fixed point $p$. 

A particular subclass of Wick-rotations which is of interest in its own right and deserves to be explored, is the class of \emph{Lie groups} $G$ equipped with left-invariant metrics, so called \emph{pseudo-Riemannian Lie groups}. If we look at a semi-simple complex Lie group $G^{\mathbb{C}}$ equipped with the left-invariant Killing form: $-\kappa$, then there are natural examples of Wick-rotations to find at the identity point, simply because there exist real forms. Moreover by the theory of semi-simple Lie groups, one may always Wick-rotate a real form $(G, -\kappa)\subset (G^{\mathbb{C}}, -\kappa)$ to a Riemannian Lie group, simply because of the existence of a Cartan involution of the Lie algebra $\mathfrak{g}$. Thus motivated by this example, then for a general pseudo-Riemannian Lie group $(G,g)$, an interesting question one may ask: 

\textit{Given a pseudo-Riemannian Lie group $(G,g)$, when can it be Wick-rotated to a Riemannian Lie group $(\tilde{G},\tilde{g})$?} 

Suppose $(G,g)$ is Wick-rotated to a Riemannian Lie group $(\tilde{G}, \tilde{g})$, then in view of the results given in (\cite{W1}, \cite{W2}, \cite{W3}), then the so called Wick-rotatable tensors restricted to $\mathfrak{g}$ must be fixed by the isometry action (induced from the metric) of some (linear) Cartan involution $\theta\in O(p,q)$ of the metric. This could for instance be the Riemann tensor $R$ (as mentioned above), and is related to the fact that $R$ can be embedded into the same complex orbit as $\tilde{R}$ (the Riemann tensor of $(\tilde{G}, \tilde{g})$ restricted to $\tilde{\mathfrak{g}}$), i.e $$O(p+q,\mathbb{C})\cdot R\ni \tilde{R}.$$

However some tensors for a left-invariant metric (for instance the Levi-Civita connection, the Riemann tensor and so on..) are very interlinked with the Lie bracket of the Lie algebra $\mathfrak{g}$. Moreover in the semi-simple case (equipped with the left-invariant Killing form) such tensors are naturally fixed by the Cartan involutions of the Lie algebra: $\theta\in Aut(\mathfrak{g})$. For example the Levi-Civita connection is given by: $\nabla_xy=\frac{1}{2}[x,y]$, thus naturally $\theta\cdot \nabla_xy=\nabla_xy$. 

The author of this paper therefore pondered about the existence of a Cartan involution: $\theta\in Aut(\mathfrak{g})$, for a general left-invariant metric (on a general Lie group $G$) which can be Wick-rotated to a Riemannian Lie group $\tilde{G}$. 

We prove an invariant for Wick-rotations of Lie groups, and give a complete answer to the question above, where we show that the answer is precisely related to the existence of a Cartan involution of the Lie algebra. Our main result of this paper is Theorem \ref{t}:

\begin{thmx}Suppose $(G,g)$ is a pseudo-Riemannian Lie group that can be Wick-rotated to another Lie group $(\tilde{G},\tilde{g})$. Then there exist a Cartan involution of $\mathfrak{g}$ if and only if there exist a Cartan involution of $\tilde{\mathfrak{g}}$. \end{thmx}

We begin this paper by defining every notion we shall use throughout, and recall the definitions of Wick-rotations in \cite{W2}. Some new definitions are also given, in particular we define a \emph{Wick-rotation of a Lie group}, and a \emph{Cartan involution} of a general Lie algebra. We also state the results we use from \cite{W3}, which makes the proofs easier to follow.
 
\begin{rem}\label{signature} In this paper a Riemannian space shall always denote the signature: \\
$(+,+,\cdots,+)$, and a Lorentzian space shall denote the signature: $(+, +,\cdots,+, -)$ and so on. The anti-isometry map $g\mapsto -g$ induces an isomorphism $O(p,q)\cong O(q,p)$. If we change signature via this anti-isometry map, then our results in this paper will be related precisely via this map as well. Moreover using a right-invariant metric instead of a left-invariant metric does not change the results of this paper.    \end{rem}

\textit{Conventions}: Throughout this paper $\kappa$ shall denote the Killing form of a Lie algebra. A product of vector spaces $V\times V$ shall often be denoted by just $V^2$. A complex Lie group shall always be denoted by the symbol: $G^{\mathbb{C}}$.

\section{Preliminaries}
\subsection{Real forms and left-invariant metrics}

In this paper a real Lie group $G$ shall be said to be an \textsl{immersive real form} of a complex Lie group $G^{\mathbb{C}}$, if there is a real immersion $G\rightarrow G^{\mathbb{C}}$ (of Lie groups) where $G^{\mathbb{C}}$ is viewed as a real Lie group, such that $\mathfrak{g}$ is embedded as a real form of $\mathfrak{g}^{\mathbb{C}}$ (the Lie algebra of $G^{\mathbb{C}}$). If the immersion is also injective then we shall call $G$ a \textsl{virtual real form}. A virtual real form $G$ which is also an embedding (i.e the image of $G$ is closed in $G^{\mathbb{C}}$), we shall say that the real form is an \textsl{embedded real form}. An embedded real form which also satisfies: $G^{\mathbb{C}}=G\cdot G^{\mathbb{C}}_0$ (abstract group product) shall be said to be a \textsl{real form}.

Note that a connected embedded real form is also a real form. All these specialised "complexifications" divide the Lie groups into different classes. For instance if $G$ is a connected semi-simple Lie group, then it is a fact that $G$ is a virtual real form if and only if $G$ is linear. 

One shall note that given any 1-connected real Lie group $G$, then we can complexify the Lie algebra via an inclusion $i$: $\mathfrak{g}\hookrightarrow \mathfrak{g}^{\mathbb{C}}$. We can find a complex 1-connected Lie group: $G^{\mathbb{C}}$ with Lie algebra $\mathfrak{g}^{\mathbb{C}}$. One can find a smooth map (of real Lie groups): $G\rightarrow G^{\mathbb{C}}$ with differential $i$, thus $G$ is an immersive real form of $G^{\mathbb{C}}$. 

We shall abuse notation and write $G\subset G^{\mathbb{C}}$ for an immersive real form.

\begin{ex} Consider the complex orthogonal group: $O(4,\mathbb{C})$, then the map: $g\mapsto I_{3,1}\overline{g}I_{3,1},$ is a conjugation map (i.e the differential is a conjugation map), where $(I_{3,1})_{ii}=+1$ for $1\leq i\leq 2, \ (I_{3,1})_{33}=-1$ and zero otherwise. The fix points of this map is just $O(1,3)$, which is an example of a real form of $O(4,\mathbb{C})$. Consider the universal covering group $G:=\widetilde{SL_2(\mathbb{R})}$ of $SL_2(\mathbb{R})$, then it is a fact that $G$ is not a virtual real form of any complex Lie group. However $G$ is an immersive real form of $SL_2(\mathbb{C})$. \end{ex}

Let $G$ be a real Lie group, then a left-invariant metric $g$ on $G$ is a pseudo-Riemannian metric satisfying: $$g_{gh}(L_{gh^*}(x_h), L_{gh^*}(y_h))=g_h(x_h, y_h), \forall g, h\in G, \ \forall x_h,y_h\in T_hG, $$ where $L_{g^*}$ is the push-forward of the translation map: $G\xrightarrow{L_g} G$: $h\mapsto gh$. Instead of writing $g_e(-,-)$ for the metric at the identity point, we simply write just $g(-,-)$. A bi-invariant metric $g$ on a real Lie group $G$ is a left-invariant metric which is also right-invariant i.e $L_g$ above is replaced with $R_g: h\mapsto hg$. 

On a real vector space $V$ a symmetric non-degenerate bilinear form $g$ shall be referred to as a \textit{pseudo-inner product}, and an inner product in the case of positive definite. A pair $(V,g)$ shall be referred to as a pseudo-inner product space (respectively inner product space). If we have a Lie algebra $\mathfrak{g}$ with a pseudo-inner product $g$ which satisfies: $$g([x,y],z)=g(x,[y,z]), \ \ x,y,z\in \mathfrak{g},$$ then $g$ shall be called \textsl{invariant}. Such a pair: $(\mathfrak{g}, g)$ is called a \textsl{quadratic Lie algebra}. For example the pair: $\Big{(} \mathfrak{sl}_2(\mathbb{R}), -\kappa \Big{)}$ is a quadratic Lie algebra, however the 3-dimensional Heisenberg Lie algebra: $\mathfrak{h}_3(\mathbb{R})$, is never a quadratic Lie algebra. We recall that an ideal $\mathfrak{I}\lhd \mathfrak{g}$ is called \textit{non-degenerate} if $\mathfrak{g}=\mathfrak{I}\oplus \mathfrak{I}^{\perp}$ w.r.t the invariant form $g$. In the case that $\mathfrak{g}$ is a reductive Lie algebra, then all ideals are in fact non-degenerate.

A \textit{holomorphic inner product} $g^{\mathbb{C}}$ on a complex vector space $V^{\mathbb{C}}$ shall be a symmetric non-degenerate complex bilinear form. The definitions of left-invariance and so on above are analogous in the case of a complex Lie group equipped with a holomorphic metric.

\begin{defn} A real Lie group $G$ equipped with a left-invariant metric $g$, denoted $(G,g)$ shall be called a \textit{pseudo-Riemannian Lie group}. If $g$ is also a Riemannian metric then the pair $(G,g)$ shall be called a \textit{Riemannian Lie group}. A complex Lie group $G^{\mathbb{C}}$ equipped with a left-invariant holomorphic metric, shall be called a \textit{holomorphic Riemannian Lie group} (or a \textit{complex Riemannian Lie group}). \end{defn}

\begin{defn} Let $(G,g_1)$ and $(H,g_2)$ be two pseudo-Riemannian Lie groups. Then $G$ is said to be isometric to $H$ if there exist a Lie group isomorphism: $G\xrightarrow{F} H$, such that $F_{*}:\mathfrak{g}\rightarrow \mathfrak{h}$ is an isomorphism of pseudo-inner product spaces: $(\mathfrak{g}, g_1)\cong (\mathfrak{h},g_2)$. The spaces are said to be locally isometric if there exist a local homomorphism $G\supset U\xrightarrow{F} V\subset H$ such that $F_{*}$ is an isomorphism of pseudo-inner product spaces: $(\mathfrak{g}, g_1)\cong (\mathfrak{h},g_2)$.   \end{defn}

The left-invariant metrics on a real Lie group $G$ are in bijections with the pseudo-inner products on the Lie algebra $\mathfrak{g}$. So we shall always work with a pseudo-inner product $g$ on the Lie algebra and induce a left-invariant metric on the Lie group by: $$g_h(x_h,y_h):=g\Big{(}L_{h^{-1}_{*}}(x_h), L_{h^{-1}_{*}}(y_h)\Big{)}, \ \ x_h, y_h\in T_hG.$$

We note that for a compact Lie group $G$, we can always complexify it to a complex Lie group: $G^{\mathbb{C}}$, such that $G\subset G^{\mathbb{C}}$ is a real form, by using the universal complexification group. In particular starting from a compact Lie group with a left-invariant metric we naturally have a candidate for a holomorphic Riemannian Lie group such that $G\subset G^{\mathbb{C}}$ is a real form. Recall that the universal complexification group of a real Lie group $G$, is a pair: $(G^{\mathbb{C}}, \eta)$, where $\eta$ is a real Lie homomorphism: $G\rightarrow G^{\mathbb{C}}$, satisfying the universal property (see for instance \cite{Neeb}). For example the pseudo-orthogonal groups: $O(p,q)$ has universal complexification group $O(p+q,\mathbb{C})$. 

\subsection{Wick-rotations of pseudo-Riemannian manifolds}

We recall some of the definitions of Wick-rotations given in \cite{W2}, and define a Wick-rotation of a pseudo-Riemannian Lie group.

\begin{defn}\label{realslice}
Given a holomorphic inner product space $(E,g^{\mathbb{C}})$. Then if $V\subset E$ is a real linear subspace for which $g:=g^{\mathbb{C}}\big{|}_V$ is non-degenerate and real valued, i.e., $g(X,Y)\in \mathbb{R},~ \forall X,Y\in V$, we will call $V$ a \textsl{real slice}. 
\end{defn}

\begin{rem} In this paper we always assume $V\subset (E,g^{\mathbb{C}})$ has the same real dimension as the complex dimension of $E$. Thus $V$ is also a real form of $E$, i.e there is a conjugation map $E\xrightarrow{\sigma} E$ with fix points $V$. We shall simply refer to $V\subset (E,g^{\mathbb{C}})$ as a real form in such a case, to mean both a real slice and a real form. \end{rem}

Thus in the definition $(V, g:=g^{\mathbb{C}}\big{|}_V)$ is a pseudo-inner product space, and if $(p,q)$ denotes the signature of $g$, then the isometry group $O(p,q)$ of $(V, g)$ is a real Lie group and is a real form of $O(p+q,\mathbb{C})$ (the isometries of $(E,g^{\mathbb{C}})$). Indeed if $\sigma$ is the conjugation map of $V$ in $E$ then note the involution $F$ of real Lie groups: $$g\mapsto \sigma g\sigma, \ g\in O(p+q,\mathbb{C}).$$ The differential of this map is a conjugation map, and $O(p,q)$ is the fix points of $F$, i.e is a real form. Such a map $F$ is often called a \textsl{real structure}.

\begin{defn}
Given a complex (holomorphic) manifold $M^{\mathbb{C}}$ with complex (holomorphic) Riemannian metric $g^{\mathbb{C}}$. If a submanifold $M\subset M^{\mathbb{C}}$ for any point $p\in M$ we have that $T_pM$ is a real slice of $(T_pM^{\mathbb{C}},g^{\mathbb{C}})$  (in the sense of  Defn. \ref{realslice}), we will call $M$ a real slice of $(M^{\mathbb{C}},g^{\mathbb{C}})$. 
\end{defn}

This definition implies that the induced metric from $M^{\mathbb{C}}$ is real valued on $M$. $M$ is therefore a pseudo-Riemannian manifold. 

\begin{defn}[Wick-related spaces]
Two pseudo-Riemannian manifolds $M$ and $\tilde{M}$ are said to be \textsl{Wick-related} if there exists a holomorphic Riemannian manifold $(M^{\mathbb{C}},g^{\mathbb{C}})$ such that $M$ and $\tilde{M}$ are embedded as real slices of $M^{\mathbb{C}}$. 
\end{defn}

\begin{defn}[Wick-rotation]
	If two Wick-related spaces (of the same real dimension) intersect at a point $p$ in $M^{\mathbb{C}}$, then we will use the term \textsl{Wick-rotation}: the manifold $M$ can be Wick-rotated to the manifold $\tilde{M}$ (with respect to the point $p$).  
\end{defn}

We now define a Wick-rotation of a pseudo-Riemannian Lie group:

\begin{defn} [Wick-rotation of a pseudo-Riemannian Lie group]
Let $G\subset G^{\mathbb{C}}\supset \tilde{G}$ be two immersive real forms which are Wick-related in $(G^{\mathbb{C}}, g^{\mathbb{C}})$ for $g^{\mathbb{C}}$ a left-invariant holomorphic metric. Then we shall say that the pseudo-Riemannian Lie group $(G,g)$ is Wick-rotated to $(\tilde{G},\tilde{g})$.  \end{defn}

Thus from the definition: $(G,g)\subset (G^{\mathbb{C}},g^{\mathbb{C}})$ is a real slice of Lie groups, and shall write $(p,q)$ for the signature of $g$. If there is another real slice $(\tilde{G}, \tilde{g})\subset (G^{\mathbb{C}},g^{\mathbb{C}})$ of Lie groups, then we shall refer to the signature of $\tilde{g}$ as $(\tilde{p},\tilde{q})$. We shall often just say a Wick-rotations of Lie groups. Note that two Lie groups which are Wick-related are also Wick-rotated at the identity point $p:=1$.

The definition implies that two Wick-rotatable metrics on real Lie groups are left-invariant themselves, and also note that a Wick-rotation of Lie groups induces in the obvious way a Wick-rotation of the identity components. Moreover the property of bi-invariance for connected groups is an invariant:

\begin{prop} Suppose $(G,g)$ is Wick-rotatable to $(\tilde{G}, \tilde{g})$ and they are both connected. Then $g(-,-)$ is bi-invariant if and only if $\tilde{g}(-,-)$ is bi-invariant.  \end{prop}
\begin{proof} The proofs given in (\cite{Milnor}, Lemma 7.1 and 7.2) also hold for pseudo-Riemannian left-invariant metrics, with $\mathfrak{o}(n)$ replaced with $\mathfrak{o}(p,q)$. Moreover if the metric $g(-,-)$ is bi-invariant, then because $ad(\mathfrak{g})\subset\mathfrak{o}(p,q)\subset \mathfrak{o}(n,\mathbb{C})$, and $ad(\mathfrak{g})^{\mathbb{C}}=ad(\mathfrak{g}^{\mathbb{C}})$ it follows that the holomorphic metric must also be bi-invariant, thus also $\tilde{g}(-,-)$. The converse is identical.  \end{proof}

Note that the property of being connected or simply connected are not necessarily preserved under a Wick-rotation. However under a Wick-rotation of real forms, then being connected is conserved.

\begin{ex}\label{wick} Let $SL_2(\mathbb{R})\subset SL_2(\mathbb{C})\supset SU(2)$ be the natural inclusions. Then they are real forms, and Wick-rotated w.r.t to the holomorphic Killing form $\kappa$ on $\mathfrak{sl}_2(\mathbb{C})$. Note that $(SL_2(\mathbb{R}), \kappa)$ is Lorentzian and $(SU(2), \kappa)$ has signature: $(-,-,-)$. \end{ex}

We also define:

\begin{defn} Let $V\subset (E,g^{\mathbb{C}})$ be a real slice. We say an involution $V\xrightarrow{\theta} V\in O(p,q)$, is a \textsl{Cartan involution} of $g:=g^{\mathbb{C}}\big{|}_{V}$, if $g_{\theta}(\cdot,\cdot):=g^{\mathbb{C}}\big{|}_{V}(\cdot,\theta(\cdot))$, is an inner product on $V$. If $\theta=1$ then $V$ is said to be a \textsl{compact real slice}, or in the case that $V$ is also a real form, then $V$ shall be said to be a \textsl{compact real form}. \end{defn}

Note the resemblance (in the definition) with a compact real form of a complex semi-simple Lie algebra and its Killing form. In the case of Lie algebras: $(\mathfrak{g}, g)\subset (\mathfrak{g}^{\mathbb{C}}, g^{\mathbb{C}})$, then a Cartan involution $\theta$ of $g$ is not necessarily a homomorphism of Lie algebras, since we do not know it they exist. We do not even know if there exist a compact real form which is also a Lie subalgebra of $(\mathfrak{g}^{\mathbb{C}},  g^{\mathbb{C}})$.  But we know if $\mathfrak{g}$ is semi-simple, and $g^{\mathbb{C}}=-\kappa$, then there exist a Cartan involution $\theta$ which is also homomorphism of the Lie algebra. 

But more generally we shall define:

\begin{defn} \label{Cinv} Let $\mathfrak{g}\subset (\mathfrak{g}^{\mathbb{C}}, g^{\mathbb{C}})$ be a real form. A Cartan involution $\theta$ of $\mathfrak{g}$ is a Cartan involution of $g:=g^{\mathbb{C}}_{|_{\mathfrak{g}}}(-,-)$ which is also a homomorphism of Lie algebras.  \end{defn}

Thus a Cartan involution of $g$ is only a linear Cartan involution of the pseudo-inner product $g$, but a Cartan involution of $\mathfrak{g}$ is a Cartan involution of $g$ which is also a homomorphism of Lie algebras. Currently at this point we only know that Cartan involutions of $\mathfrak{g}$ exist when $\mathfrak{g}$ is abelian or $\mathfrak{g}$ is semi-simple equipped with the Killing form: $-\kappa$. One shall note that there are examples where they do not exist, indeed by changing the sign to: $\kappa$, then it is straight forward to show that there are no Cartan involutions of $\mathfrak{g}$.

\begin{defn} Two real forms $V$ and $\widetilde{V}$ of $E$ are said to be compatible if their conjugation maps commute, i.e $[\sigma,\tilde{\sigma}]=0$.  \end{defn}

Often we shall refer to a pair $(V, \tilde{V})$ as a compatible pair- to mean that the spaces are compatible.

We recall from \cite{W2}, that if $(E,g^{\mathbb{C}})$ is a holomorphic inner product space, and $V, \tilde{V}$ and $W$ are real forms such that $W$ is a compact real form (i.e of Euclidean signature), then if they are pairwise compatible, the triple: $\Big{(}V,\tilde{V}, W \Big{)}$, is said to be a \textsl{compatible triple}. Note that Example \ref{wick} is an example of a compatible triple: $$\Big{(}V:=\mathfrak{sl}_2(\mathbb{R}), \tilde{V}:=\mathfrak{su}(2), W:=\mathfrak{su}(2)\Big{)}.$$ 
\\

We shall call the eigenspace decomposition of a Cartan involution: $\theta$, for the Cartan decomposition. 

\begin{rem} \label{Ocon}By the uniqueness of a signature associated to a pseudo-inner product $g$ then all Cartan involutions of $g$ are conjugate in $O(p,q)$. In fact given two Cartan involutions: $\theta_j$ ($j=1,2$) then $g\mapsto \theta_jg\theta_j$ is a global Cartan involution of $O(p,q)$. Thus if $g\theta_1g^{-1}=\theta_2$ for some $g\in O(p,q)$, then writing $g=k_2e^x$, where $k_2$ commutes with $\theta_2$ and $x\in \mathfrak{o}(p,q)$, we obtain $\theta_1=e^x\theta_2e^{-x}$, and therefore $\theta_1, \theta_2$ are conjugate by an element $g\in O(p,q)_0$.\end{rem}

Suppose now we have a Wick-rotation of two real Lie groups: $(G,g)\subset (G^{\mathbb{C}}, g^{\mathbb{C}})\supset (\tilde{G},\tilde{g})$. Let $\theta\in O(p,q)$ be a Cartan involution of the metric $g$, and let $W$ denote the corresponding unique compact real form associated with $\theta$, i.e $W:=V_+\oplus iV_-,$ where $\mathfrak{g}=V_+\oplus V_-$ is the Cartan decomposition. Then by \cite{W3} it is possible to find a real form $\tilde{V}\subset \mathfrak{g}^{\mathbb{C}}$ (as vector spaces) and a linear isomorphism: $\tilde{V}\xrightarrow{\phi}\tilde{\mathfrak{g}}$ such that $\phi^{\mathbb{C}}\in O(n,\mathbb{C})$, and $(\mathfrak{g}, \tilde{V}, W)$ is a compatible triple. So consider the triple: $\Big{(} \mathfrak{o}(p,q), \mathfrak{o}(\tilde{p},\tilde{q}), \mathfrak{o}(n)\Big{)}$, of Lie algebras of the isometry groups associated with the compatible triple $(\mathfrak{g}, \tilde{V}, W)$. 

Then the following straightforward result is important to note:

\begin{lem} [\cite{W2}, Lemma 3.6] \label{trippel} The triple of real forms: $\Big{(} \mathfrak{o}(p,q), \mathfrak{o}(\tilde{p},\tilde{q}), \mathfrak{o}(n)\Big{)}$, embedded into $\mathfrak{o}(n,\mathbb{C})$ is a compatible triple of Lie algebras. \end{lem}

Thus we note that up to an isometry $g\in O(n,\mathbb{C})$ we may assume our two Lie algebras $\mathfrak{g}$ and $\tilde{\mathfrak{g}}$ (viewed as a vector space) form a compatible triple with a compact real form $W\subset (\mathfrak{g}^{\mathbb{C}}, g^{\mathbb{C}})$. 

\subsection{Real GIT on compatible representations} \label{compatiblereps}

In this section we recall some definitions and results of \cite{W3} that we shall use. We consider certain type of groups here. When considering a real form: $G\subset G^{\mathbb{C}}$, then $G^{\mathbb{C}}$ shall be of type linearly complex reductive, and $G$ should either be linearly real reductive, or in the case where $G^{\mathbb{C}}\subset GL(V^{\mathbb{C}})$ is defined over $\mathbb{R}$, the real points: $G:=GL(V)\cap G^{\mathbb{C}}$. This is the assumptions in the paper \cite{W3}. Thus we may for instance use the pseudo-orthogonal group $O(p,q)\subset O(n,\mathbb{C})$ defined as the isometry group of some pseudo-inner product space: $(V,g)\subset (V^{\mathbb{C}},g^{\mathbb{C}})$. A compact real form of $G^{\mathbb{C}}$ shall always be denoted by $U$.

Let $G\subset GL(V)$ be such a group. A Cartan involution $\theta$ of $\mathfrak{g}$ is now a Cartan involution in the sense of a reductive Lie algebra. Recall that this means that $\theta$ is the restriction of a Cartan involution of $\mathfrak{gl}(V)$. In view of Definition \ref{Cinv},  $\theta$ is a Cartan involution of $(\mathfrak{g}, g)$, with $g=\lambda \kappa\oplus B$ $(\lambda<0)$, where $\kappa$ is the Killing form on $[\mathfrak{g}, \mathfrak{g}]$ and $B$ a pseudo-inner product on $\mathfrak{z}(\mathfrak{g})$. We refer to for example \cite{W3} or \cite{RS} where such Cartan involutions are considered in more detail. A global Cartan involution $\Theta$ with $d\Theta=\theta$ of $G$ always exist for such groups. For example the class of linear semisimple Lie groups of finitely many connected components (fcc) are one such class.

\begin{defn} Let $G\subset G^{\mathbb{C}}\supset \tilde{G}$ be two real Lie subgroups of a complex Lie group such that the real Lie algebras are real forms of $\mathfrak{g}^{\mathbb{C}}$. Then we say $G$ and $\tilde{G}$ are \textsl{compatible} if the Lie algebras are compatible. \end{defn}

\begin{defn} Let $G\subset G^{\mathbb{C}}\supset \tilde{G}$ and $U\subset G^{\mathbb{C}}$ be real Lie subgroups of a complex Lie group such that the real Lie algebras are real forms of $\mathfrak{g}^{\mathbb{C}}$. Moreover assume $U$ is compact. Then we say $\Big{(}G, \tilde{G}, U\Big{)}$ is a \textsl{compatible triple} if the Lie algebras are pairwise compatible. \end{defn}

If we use Lemma \ref{trippel}, in the context of Wick-rotations (see the previous section), then the triple of isometry groups: $\Big{(} O(p,q), O(\tilde{p},\tilde{q}), O(n)\Big{)}$ form a compatible triple when the pseudo-inner product spaces they are isometries of, form a compatible triple.  

\begin{defn} [\cite{RS}] Let $G\xrightarrow{\rho^G_V} GL(V)$ be a real representation, then $\rho^G_V$ is said to be a \emph{balanced representation} if there exist an involution $V\xrightarrow{\theta} V$, and a global Cartan involution: $G\xrightarrow{\Theta} G$ such that: $$\Big{(}\forall g\in G \Big{)}\Big{(}\rho^G_V(\Theta(g))=\theta\circ\rho^G_V(g)\circ\theta\Big{)}.$$ \end{defn}

Thus if we have an involution $\theta$ of $V$ balancing our action, then w.r.t the global Cartan involution $\Theta$ of $G$ with Cartan decomposition: $G=Ke^{\mathfrak{p}}$, there exist a pseudo-inner product $g(-,-)$ on $V$ such that $\theta$ is a Cartan involution of $g(-,-)$, and the inner product $g_{\theta}(-,-):=g(-,\theta(-))$ is $K$-invariant. Let $\mathcal{M}(G,V)$ denote the minimal vectors of our action, i.e those $v\in V$ satisfying: $||g\cdot v||\geq ||v||$ for all $g\in G$, where $||v||^2:=g_{\theta}(v,v)$. Then if $V=V_+\oplus V_-$ is the Cartan decomposition, we naturally have $V_+\cup V_-\subset \mathcal{M}(G,V)$. The Cartan involutions of $g(-,-)$ which are conjugate by the action of $G$ to $\theta$ are defined as the \textsl{inner Cartan involutions} of $g(-,-)$.

A complex action: $\rho^{\mathbb{C}}$ of $G^{\mathbb{C}}$ acting on $V^{\mathbb{C}}$ is said to be a complexified action of a real action $\rho^G_V$ if $\rho^{\mathbb{C}}(g)(v)=\rho(g)(v)$ for all $g\in G$ and $v\in V$.

\begin{defn} \label{triple} Let $G\subset G^{\mathbb{C}}\supset \tilde{G}$ be real forms, and $G\xrightarrow{\rho^G_V} GL(V)$ and $\tilde{G}\xrightarrow{\rho^{\tilde{G}}_{\tilde{V}}} GL(\tilde{V})$ be real representations of Lie groups. Suppose $G^{\mathbb{C}}\xrightarrow{\rho^{\mathbb{C}}}GL(V^{\mathbb{C}})$ is a complexified action of both $\rho^G_V$ and $\rho^{\tilde{G}}_{\tilde{V}}$. Then we say that $\rho^G_V$ is \textsl{compatible} with $\rho^{\tilde{G}}_{\tilde{V}}$, if the following two criterions are fulfilled:
\begin{enumerate}

\item{} $G$ and $\tilde{G}$ are compatible real forms of $G^{\mathbb{C}}$.
\item{} $V$ and $\tilde{V}$ are compatible real forms of $V^{\mathbb{C}}$.
\end{enumerate}
\end{defn}

\begin{defn} \label{comp} Let $\rho^G_V, \rho^{\tilde{G}}_{\tilde{V}}$ and $\rho^U_W$ be pairwise compatible representations, where $U\subset G^{\mathbb{C}}$, is a compact real form. Then the triple: $\Big{(}\rho^G_V, \rho^{\tilde{G}}_{\tilde{V}}, \rho^U_W\Big{)}$ is said to be a \textsl{compatible triple}. \end{defn}

If we have such a compatible triple, then all the real actions in the triple are balanced, and we can choose pseudo-inner products $g(-,-)$ and $\tilde{g}(-,-)$ on $V$ and $\tilde{V}$ respectively, in such a way that they restrict from the same Hermitian form on $V^{\mathbb{C}}$. Moreover if $\tau$ denotes the conjugation map of $W$ in $V^{\mathbb{C}}$ then it restricts to Cartan involutions: $\theta$ (of $g$) and $\tilde{\theta}$ (of $\tilde{g}$). The Cartan involutions also balance the real actions respectively. In particular the inner products $g_{\theta}$ and $\tilde{g}_{\tilde{\theta}}$ both restrict from the $U$-invariant Hermitian inner product $H(-,\tau(-))$. The minimal vectors satisfy: $$\mathcal{M}(G,V)\subset \mathcal{M}(G^{\mathbb{C}}, V^{\mathbb{C}})\supset \mathcal{M}(\tilde{G},\tilde{V}), \ \ \ W\subset \mathcal{M}(G^{\mathbb{C}},V^{\mathbb{C}}).$$ Denote the Cartan decompositions by $V=V_+\oplus V_-$ and $\tilde{V}=\tilde{V}_+\oplus \tilde{V}_-$ respectively.

\begin{defn} Let $\Big{(}\rho^G_V, \rho^{\tilde{G}}_{\tilde{V}}\Big{)}$ be a compatible pair. Suppose $v\in V$ and $\tilde{v}\in \tilde{V}$ are such that $\tilde{v}\in G^{\mathbb{C}}v$, then we shall say that $Gv$ is \textsl{compatible} with $\tilde{G}\tilde{v}$. We write $Gv\sim \tilde{G}\tilde{v}$. \end{defn}

It is important to note the following result:

\begin{thm} [\cite{W3}] \label{main} Let $\Big{(}\rho^G_V, \rho^{\tilde{G}}_{\tilde{V}}, \rho^U_W\Big{)}$ be a compatible triple. Suppose $v\in V$ and $\tilde{v}\in \tilde{V}$ are such that: $\tilde{G}\tilde{v}\sim Gv$. Then $Gv\cap V_+\neq \emptyset$ (respectively $Gv\cap V_-\neq\emptyset$) if and only if $\tilde{G}\tilde{v}\cap \tilde{V}_+\neq\emptyset$ (respectively $\tilde{G}\tilde{v}\cap \tilde{V}_-\neq \emptyset$). \end{thm}

Observe that if there exist $v_+\in Gv$, then $\theta(v_+)=v_+$, i.e if $g\in G$ is such that $g\cdot v=v_+$, then there is an inner Cartan involution $\theta'$ of $g(-,-)$ such that $\theta'(v)=v$ using $g$.

We shall also state the following important result:

\begin{thm} [\cite{W3}] \label{54} Let $(\rho^G_{V}, \rho^U_{W})$ be a compatible pair. Let $v\in V$, then the following statements are equivalent: 

\begin{enumerate}[label=\Alph*]
\item{}There exist $w\in W$ such that $Uw\sim Gv$. 
\item{}There exist an inner Cartan involution $V\xrightarrow{\theta} V$ such that $\theta(v)=v$.
\item{}There exist $w\in W$ such that $Uw\cap Gv\neq \emptyset$. 
\end{enumerate} 
\end{thm}

In fact if there is a $w\in W$ and $v\in V$ such that $Uw\sim Gv$ then: $$\emptyset\neq Uw\cap Gv=Gv\cap \mathcal{M}(G,V)=Kv,$$ where $K=U\cap G$. 

A worked out example of compatible representations is given in the next section in the context of Wick-rotations of Lie groups.

\subsection{The isometry action on bilinear maps into the Lie algebra}\label{bilaction}
In this section we shall consider the action that we are going to use to prove our main result of this paper. We shall explain in detail that under a Wick-rotation, the isometry groups of the pseudo-inner product spaces induces compatible representations (see Defn. Section \ref{compatiblereps}). 

Suppose we have a Wick-rotation of pseudo-Riemannian Lie groups: \\
$(G,g)\subset (G^{\mathbb{C}},g^{\mathbb{C}})\supset (\tilde{G}, \tilde{g})$. As we have seen we can choose a map $g\in O(n,\mathbb{C})$ such that we obtain a compatible triple: $(\mathfrak{g}, \tilde{V}, W)$, with $\tilde{V}:=g(\tilde{\mathfrak{g}})$. We shall denote $\tilde{g}$ also for the pseudo-inner product on $\tilde{V}$ restricted from $g^{\mathbb{C}}$. We can choose a pseudo-orthonormal basis: $\{e_1,\dots, e_p, \dots ,e_n\}$ (of $g$) and similarly $\{\tilde{e}_1,\dots, \tilde{e}_{\tilde{p}},\dots, \tilde{e}_n\}$ (of $\tilde{g}$), such that $W$ is the real span of both the sets: $Y:=\{e_1,\dots, e_p, ie_{p+1} \dots ie_n\}$ and $\tilde{Y}:=\{\tilde{e}_1,\dots, \tilde{e}_{\tilde{p}},i\tilde{e}_{\tilde{p}+1},\dots ,i\tilde{e}_n\}$. Denote the corresponding Cartan involutions by $\theta$ (of $g$) and $\tilde{\theta}$ (of $\tilde{g}$). Note that $Y$ and $\tilde{Y}$ are both an orthonormal basis of $g^{\mathbb{C}}$.

Consider the complex isometry action of $O(n,\mathbb{C})$ on $\mathfrak{g}^{\mathbb{C}}$ by $g\cdot x:=g(x)$. This action restricts to the real isometry actions of $O(p,q)$ on $\mathfrak{g}$ and $O(\tilde{p},\tilde{q})$ on $V$ respectively. Let $\mathcal{V}$ and $\tilde{\mathcal{V}}$ denote the real vector spaces of bilinear maps: $\mathfrak{g}^2\rightarrow \mathfrak{g}$ (respectively $\tilde{V}^2\rightarrow \tilde{V}$). Thus $\mathcal{V}\subset \mathcal{V}^{\mathbb{C}}\supset \tilde{\mathcal{V}}$ are real forms, where $\mathcal{V}^{\mathbb{C}}$ is the complex vector space of complex bilinear maps: $({\mathfrak{g}^{\mathbb{C}}})^2\rightarrow \mathfrak{g}^{\mathbb{C}}.$ The complex isometry action naturally extends to a complex action of $O(n,\mathbb{C})$ on $b\in \mathcal{V}^{\mathbb{C}}$, by $$(g\cdot b)(x,y):=g\Big{(}b(g^{-1}(x),g^{-1}(y))\Big{)}, \ x,y\in \mathfrak{g}^{\mathbb{C}}, \ g\in O(n,\mathbb{C}).$$ Note that the action again restricts to action of the real isometry groups on $\mathcal{V}$ and $\tilde{\mathcal{V}}$ respectively. Denote the real actions by $\rho$ and $\tilde{\rho}$ respectively. The Cartan involution $\theta\in O(p,q)$ (respectively $\tilde{\theta}\in O(\tilde{p},\tilde{q})$) naturally extends to an involution of $\mathcal{V}$ (respectively $\tilde{\mathcal{V}}$), by the action: $\rho(\theta)$ (respectively $\tilde{\rho}(\tilde{\theta})$). The holomorphic inner product $g^{\mathbb{C}}$ extends naturally to a holomorphic inner product: $\textbf{g}^{\mathbb{C}}$, by defining: $$\textbf{g}^{\mathbb{C}}(b_1,b_2):=\sum^n_j g^{\mathbb{C}}\Big{(}b_1(y_j,y_j), b_2(y_j,y_j)\Big{)}.$$ Observe that if we change basis w.r.t to $\tilde{Y}$ instead then we obtain the same holomorphic inner product. Indeed this follows since we can find $g\in O(n,\mathbb{C})$ sending $Y\mapsto \tilde{Y}$. It is easy to check that $\mathcal{V}\subset (\mathcal{V}^{\mathbb{C}},\textbf{g}^{\mathbb{C}})\supset \tilde{\mathcal{V}}$ are real slices. Similarly if we define $\mathcal{W}$ to be all bilinear maps: $W^2\rightarrow W$, then by construction $\mathcal{W}$ is a compact real form of $(\mathcal{V}^{\mathbb{C}},\textbf{g}^{\mathbb{C}})$. Observe that the three real forms form a compatible triple in $\mathcal{V}^{\mathbb{C}}$. Therefore the actions form a compatible triple (see Section \ref{compatiblereps}). There is a natural choice of $O(n)$-invariant Hermitian inner product on $\mathcal{V}^{\mathbb{C}}$, namely: $H:=\textbf{g}^{\mathbb{C}}(\cdot, \mathcal{T}(\cdot))$, where $\mathcal{T}$ is the conjugation map of $\mathcal{W}$. This Hermitian inner product restricts to inner products on $\mathcal{V},\tilde{\mathcal{V}}$ and $\mathcal{W}$. Observe that the inner Cartan involutions of $\rho$ (respectively $\tilde{\rho}$) are those conjugate to $\rho(\theta)$ (respectively $\tilde{\rho}(\tilde{\theta})$). 

\subsection{Wick-rotatable tensors of pseudo-Riemannian manifolds}
For a Wick-rotation of Lie groups it is worth noting that the action in the previous section is just an example of a tensor action of $O(n,\mathbb{C})$ on a general tensor space of finite  form: $$\mathcal{V}^{\mathbb{C}}:=\bigoplus_{k,m}\Big{(}\Big{(}\bigotimes_{i=1}^k \mathfrak{g}^{\mathbb{C}}\Big{)}\bigotimes\Big{(}\bigotimes_{i=1} ^m (\mathfrak{g}^{\mathbb{C}})^*\Big{)}\Big{)},$$ induced from the isometry action of the holomorphic metric $g^{\mathbb{C}}$. Analogously we define: $$\mathcal{V}:=\bigoplus_{k,m}\Big{(}\Big{(}\bigotimes_{i=1}^k \mathfrak{g}\Big{)}\bigotimes\Big{(}\bigotimes_{i=1} ^m \mathfrak{g}^*\Big{)}\Big{)}, \ \ \ \tilde{\mathcal{V}}:=\bigoplus_{k,m}\Big{(}\Big{(}\bigotimes_{i=1}^k \tilde{\mathfrak{g}}\Big{)}\bigotimes\Big{(}\bigotimes_{i=1} ^m \tilde{\mathfrak{g}}^*\Big{)}\Big{)}.$$ The real isometry groups: $O(p,q)$ (respectively $O(\tilde{p},\tilde{q})$) restrict to acting on $\mathcal{V}$ (respectively $\tilde{\mathcal{V}}$). 

More generally for a Wick-rotation of pseudo-Riemannian manifolds: 
$$(M,g)\subset (M^{\mathbb{C}}, g^{\mathbb{C}})\supset (\tilde{M},\tilde{g}),$$ at a common point $p\in M\cap \tilde{M}$, then by replacing $\mathfrak{g}$ with $T_pM$ (respectively $\tilde{\mathfrak{g}}$ with $T_p\tilde{M}$), and $\mathfrak{g}^{\mathbb{C}}$ with $T_pM^{\mathbb{C}}$, we obtain the induced tensor action on real forms: $\mathcal{V}\subset \mathcal{V}^{\mathbb{C}}\supset \tilde{\mathcal{V}}$. 

One shall note that the metrics, Cartan involutions all extend naturally to these spaces via the tangent spaces. Moreover if $g\in O(n,\mathbb{C})$ is such that $T_pM$ and $g(T_p\tilde{M})$ form a compatible triple with a compact real form $W\subset T_pM^{\mathbb{C}}$, then naturally also $\mathcal{V}$ and $g\cdot\tilde{\mathcal{V}}$ form a compatible triple with $\mathcal{W}:=\bigoplus_{k,m}\Big{(}\Big{(}\bigotimes_{i=1}^k W\Big{)}\bigotimes\Big{(}\bigotimes_{i=1} ^m W^*\Big{)}\Big{)}$. 

For example the induced action of $O(n,\mathbb{C})$ on $End(T_pM^{\mathbb{C}})$ given by conjugation: $g\cdot f:=gfg^{-1}$ is just the tensor action: $g\cdot (v_1\otimes v_2):=g(v_1)\otimes g(v_2)$, for an $O(n,\mathbb{C})$-module isomorphism: $End(T_pM^{\mathbb{C}})\cong T_pM^{\mathbb{C}}\otimes T_pM^{\mathbb{C}}$. For a more detailed explanation of this example, and on the tensor action in general we refer to \cite{W3}.

Consider the action in the previous section for instance, then one should observe that the complex Lie bracket $v:=[-,-]$ of $\mathfrak{g}^{\mathbb{C}}$ is a vector in $\mathcal{V}$, but also there is a $g\in O(n,\mathbb{C})$ such that $\tilde{v}:=g\cdot v\in g\cdot\tilde{\mathcal{V}}$, i.e $v$ and $\tilde{v}$ lie in the same complex orbit: $O(n,\mathbb{C})v\ni\tilde{v}$, in such a way that $O(p,q)v\sim O(\tilde{p},\tilde{q})\tilde{v}$ are compatible real orbits. 

Thus it useful to define for general tensors $v\in \mathcal{V}$ and $\tilde{v}\in \tilde{\mathcal{V}}$:

\begin{defn} [\cite{W3}] Let $(M,g)$ and $(\tilde{M},\tilde{g})$ be two Wick-rotatable pseudo-Riemannian manifolds at a common point $p$. Then two tensors $v\in \mathcal{V}$ and $\tilde{v}\in \tilde{\mathcal{V}}$ are said to be \textsl{Wick-rotatable} at $p$, if they lie in the same $O(n,\mathbb{C})$-orbit, i.e $$O(n,\mathbb{C})v=O(n,\mathbb{C})\tilde{v}.$$  \end{defn}

One should note the subset of Wick-rotatable tensors consisting of those in the intersection: $v\in \mathcal{V}\cap\tilde{\mathcal{V}}$. Then there is a map $g\in O(n,\mathbb{C})$, such that $v$ and $g\cdot v\in g\cdot \tilde{\mathcal{V}}$ lie in the same complex orbit such that $O(p,q)v\sim O(\tilde{p},\tilde{q})\tilde{v}$ are compatible real orbits. More generally if $v$ and $\tilde{v}$ are Wick-rotatable i.e by definition $O(n,\mathbb{C})v=O(n,\mathbb{C})\tilde{v}$, then also $O(n,\mathbb{C})v=O(n,\mathbb{C})g\cdot \tilde{v}$. The main point is to be able to embed the vectors into the same complex orbit, such that we may apply the results of Section \ref{compatiblereps}.
\\

Let $(M,g)$ be a pseudo-Riemannian manifold of signature $(p,q)$, and $\theta\in O(p,q)$ be a Cartan involution of $g(-,-)$. Consider the isometry tensor action of $O(p,q)$ on $\mathcal{V}$ as above: $$O(p,q)\xrightarrow{\rho^{O(p,q)}_{\mathcal{V}}} GL(\mathcal{V}).$$ Then $\theta$ naturally extends to an involution $\Theta:=\rho^{O(p,q)}_{\mathcal{V}}(\theta)$ on $\mathcal{V}$, and the metric naturally induces a pseudo-inner product: $\textbf{g}(-,-)$  on $\mathcal{V}$ such that $\Theta$ is a Cartan involution. Let now $R\in \mathcal{V}$ be the Riemann tensor of $M$ at $p$ for $\mathcal{V}$ some tensor space. For example $R$ may be considered as a multilinear form into $T_pM$: $T_pM^3\rightarrow T_pM$, where the action is given by: $$(g\cdot R)(x,y,z):=g\Big{(}R(g^{-1}(x), g^{-1}(y), g^{-1}(z))\Big{)}, \ x,y,z\in T_pM, \ g\in O(p,q).$$ Another approach is to consider $R$ as a map in $End(\mathfrak{o}(p,q))\subset End\Big{(}End(T_pM)\Big{)}$ at the point $p$, where the action is given by: $$(g\cdot R)(X):=gR(g^{-1}Xg)g^{-1}, \ X\in \mathfrak{o}(p,q), \ g\in O(p,q).$$ The Riemann tensor at $p$ is viewed in this way for instance in \cite{W1}. One may show that these two actions are equivalent up to an $O(p,q)$-module isomorphism, by identifying the spaces with the tensor space: $T_pM\otimes T_pM\otimes T_pM\otimes T_pM$. 

We also recall the following definition:

\begin{defn}\label{RPM} If there exist a Cartan involution $\Theta$ such that $\Theta(R)=R$ (respectively $\Theta(R)=-R$), then the space $(M,g)$ at $p$ is called \textit{Riemann purely electric} (RPE) (respectively \textit{Riemann purely magnetic} (RPM)). If there is such a $\Theta$ for the Weyl tensor at $p$, then $(M,g)$ at $p$ is called \textit{purely electric} (PE) (respectively \textit{purely magnetic} (PM)). \end{defn}

Any Riemannian space $(M,g)$ is RPE at any point $p\in M$, since the identity map $\theta:=1_{T_pM}$ is a Cartan involution of the metric $g$ at any point, thus the Cartan involution extended to tensors: $\mathcal{V}$ is also the identity map, i.e $\Theta(R)=R$.

The Levi-Civita connection $\nabla$ of a real slice of a holomorphic Riemannian manifold $(M,g)\subset (M^{\mathbb{C}},g^{\mathbb{C}})$ at $p\in M$, restricts from the complex Levi-Civita connection: $\nabla^{\mathbb{C}}$ at $p$ of the complex manifold $M^{\mathbb{C}}$. Thus the real Riemann curvature tensor $R$ (of $M$) at $p$ restricts from the complex Riemann curvature tensor $R^{\mathbb{C}}$ of $M^{\mathbb{C}}$ (at $p$). Moreover if $ric_g$ denotes the real Ricci curvature: $T_pM^2\rightarrow \mathbb{R}$, defined by: $$ric_g(x,y):=Tr\Big{(}z\mapsto R(z,y)(x)\Big{)},$$ then using a real basis of $T_pM$ also for $T_pM^{\mathbb{C}}$ we see that restricting the complex Ricci curvature: $ric_{g^{\mathbb{C}}}$ on $M^{\mathbb{C}}$ to $T_pM$ we get $ric_g$. Similarly the real Ricci operator: $$Ric_g\in End(T_pM),\ \ g_p(Ric_g(x),y)=ric_g(x,y),$$ restricts form the complex Ricci curvature operator of $M^{\mathbb{C}}$ (at $p$). 

This means that in terms of Wick-rotations of pseudo-Riemannian manifolds at a common point $p$: $(M,g)\subset (M^{\mathbb{C}}, g^{\mathbb{C}})\supset (\tilde{M},\tilde{g}),$ we see that the pairs of tensors: $$(\nabla, \tilde{\nabla}), \ (R, \tilde{R}), \ (ric_g, ric_{\tilde{g}}), \ (Ric_g, Ric_{\tilde{g}}),$$ are examples of Wick-rotatable tensors (at $p$) in the intersection $\mathcal{V}\cap\tilde{\mathcal{V}}$. The induced isometry action of $O(n,\mathbb{C})$ on these tensors (induced from the isometry action of the metric) can be naturally seen as the actions: $$(g\cdot \nabla)(x,y):=g(\nabla_{g^{-1}x}g^{-1}y), \ \ (g\cdot R)(x,y,z):=g\Big{(}R(g^{-1}x,g^{-1}y,g^{-1}z)\Big{)}$$ and $$(g\cdot ric_g)(x,y):=ric_g(gx,gy), \ \ (g\cdot Ric_g)(x):=(g\circ Ric_g \circ g^{-1})(x).$$ 

An immediate new result is the following:

\begin{thm}\label{a} Let $(M,g)\subset (M^{\mathbb{C}}, g^{\mathbb{C}})\supset (\tilde{M},\tilde{g})$ be a Wick-rotation at a common point $p\in M\cap \tilde{M}$. Assume $(\tilde{M}, \tilde{g})$ is a Riemannian space. Then the following statements hold:
\begin{enumerate}
\item{} There exist a Cartan involution $\theta$ of $g$ such that $\nabla_{\theta(x)}\theta(y)=\theta(\nabla_xy)$ for all $x,y\in T_pM$.
\item{} There exist a Cartan involution $\theta$ of $g$ such that $ric_g\Big{(}\theta(x),\theta(y)\Big{)}=ric_g(x,y)$ for all $x,y\in T_pM$.
\item{} There exist a Cartan involution $\theta$ of $g$ such that $[\theta, Ric_g]=0$.
\item{} There exist a Cartan involution $\theta$ of $g$ such that $R\Big{(}\theta(x),\theta(y)\Big{)}\Big{(}\theta(z)\Big{)}=\theta\Big{(}R(x,y)(z)\Big{)}$ for all $x,y,z\in T_pM$. Thus $(M,g)$ is (RPE) at $p$.
\end{enumerate}
 \end{thm} 
\begin{proof} It is enough to spell out the proof for the first case, as the other cases are identical. Let $v:=\nabla\in \mathcal{V}$ and $\tilde{v}:=\tilde{\nabla}\in \tilde{\mathcal{V}}$, and consider the isometry tensor action as above. The vectors $v$ and $\tilde{v}$ are Wick-rotatable, thus up to a map $g\in O(n,\mathbb{C})$ we can assume the real actions are compatible, and that $v$ and $\tilde{v}$ lie in the same complex orbit, such that the real orbits: $O(p,q)v\sim O(\tilde{p},\tilde{q})$ are compatible. The result now follows from Theorem \ref{54}, since $O(\tilde{p},\tilde{q})=O(n)$ is a compact real form of $O(n,\mathbb{C})$. \end{proof}

One shall note that Case (4) of the theorem is proved in \cite{W1}. We shall strengthen Theorem \ref{a} for Wick-rotations of pseudo-Riemannian Lie groups in the last section of the paper, by proving that a Cartan involution of $g$ may be chosen to be a homomorphism of Lie algebras.

\section{An invariant of Wick-rotation of Lie groups}
In this section we shall prove the main theorem of the paper, which is an invariance result based on the existence of a Cartan involution of the Lie algebras (Defn. \ref{Cinv}).

Let $(G,g)\subset (G^{\mathbb{C}},g^{\mathbb{C}})\supset (\tilde{G},\tilde{g})$ be a Wick-rotation of Lie groups. Consider the action in Section \ref{bilaction} and following the notation there, then by our preparations, the main result is now easily deducible:

\begin{thm} \label{t} Suppose $(G,g)$ is a pseudo-Riemannian Lie group that can be Wick-rotated to another Lie group $(\tilde{G},\tilde{g})$. Then there exist a Cartan involution of $\mathfrak{g}$ if and only if there exist a Cartan involution of $\tilde{\mathfrak{g}}$.   \end{thm}
\begin{proof} Consider the group action and the notation as in Section \ref{bilaction}. Thus if $v:=[-,-]$ is the Lie bracket of $\mathfrak{g}^{\mathbb{C}}$ then $v\in \mathcal{V}$ and restricts to the Lie bracket of $\tilde{\mathfrak{g}}$. We can choose $g\in O(n,\mathbb{C})$ such that $g\cdot v\in \tilde{\mathcal{V}}$, i.e $v$ and $\tilde{v}:=g\cdot v$ lie in the same complex orbit, thus $O(p,q)v\sim O(\tilde{p}, \tilde{q})\tilde{v}$ are compatible real orbits. Suppose $\theta$ is a Cartan involution of $\mathfrak{g}$, and denote $\mathcal{V}=\mathcal{V}_+\oplus \mathcal{V}_-$ (respectively $\tilde{\mathcal{V}}=\tilde{\mathcal{V}}_+\oplus \tilde{\mathcal{V}}_-$ ) for the Cartan decomposition w.r.t to $\rho(\theta)$ (respectively $\tilde{\rho}(\tilde{\theta})$). Then the action of $\theta$ on $v$ fixes $v$, i.e $\rho(\theta)(v):=\theta\cdot v=v$, thus $v\in \mathcal{V}_+$. Hence the real orbit: $O(p,q)v$ intersects $\mathcal{V}_+$. But then by Theorem \ref{main}, it follows that there exist also $\tilde{v}'\in \tilde{\mathcal{V}}_+\cap O(\tilde{p},\tilde{q})\tilde{v}$. Therefore choose $h\in O(\tilde{p},\tilde{q})$ such that $h\cdot \tilde{v}=\tilde{v}'$. By conjugating $\tilde{\rho}(\tilde{\theta})$ by $h$ we obtain a Cartan involution $\tilde{\theta}'$ of $\tilde{g}$ such that $\tilde{\theta}'\cdot \tilde{v}=\tilde{v}$. Finally since $\tilde{V}:=g(\tilde{\mathfrak{g}})$ for some $g\in O(n,\mathbb{C})$ then the Cartan involution $g^{-1}\tilde{\theta}'g$ fixes $v$, i.e is a Cartan involution of $\tilde{g}$ and a homomorphism of Lie algebras. The converse is symmetric. The theorem is proved.    \end{proof}

We find it useful to define for future exploration:

\begin{defn} A property of a pseudo-Riemannian Lie group $(G,g)$ is said to be \textit{Wick-rotatable} if it is an invariant under a Wick rotation of Lie groups. \end{defn}

\begin{cor} The existence of a Cartan involution of $\mathfrak{g}$ is Wick-rotatable. \end{cor}

Other Wick-rotatable properties include (see for example \cite{Helgason} on complexification of real Lie algebras): being semi-simple, abelian, nilpotent, solvable, reductive. Note that being simple, is not Wick-rotatable, indeed as an example consider the Lie group $O(1,3)$ with the left-invariant metric being the Killing form. Then $\mathfrak{o}(1,3)$ is simple, but we may Wick-rotate $O(1,3)$ to $O(2,2)$ which is semi-simple but not simple, as $\mathfrak{o}(2,2)\cong \mathfrak{sl}_2(\mathbb{R})^2$ (two copies).
\\

We can now answer the question for when an arbitrary left-invariant metric can be Wick-rotated to a Riemannian left-invariant metric. One should compare the result with semi-simple Lie groups equipped with the left-invariant Killing form: $g:=-\kappa$.

\begin{cor} \label{Riem} Suppose $(G,g)\subset (G^{\mathbb{C}},g^{\mathbb{C}})$ is a real slice of Lie groups. Then $(G,g)$ can be Wick-rotated to a Riemannian Lie group $(\tilde{G}, \tilde{g})$ if and only if there exist a Cartan involution of $\mathfrak{g}$. \end{cor}
\begin{proof} $(\Rightarrow)$. The identity map $\tilde{\mathfrak{g}}\xrightarrow{1} \tilde{\mathfrak{g}}$ is a Cartan involution of $\tilde{\mathfrak{g}}$. Thus by Theorem \ref{t} the direction follows. Conversely suppose $\theta$ is a Cartan involution of $\mathfrak{g}$, and write $\mathfrak{g}=\mathfrak{k}\oplus \mathfrak{p}$, for the Cartan decomposition. Then is is not difficult to show that $\tilde{\mathfrak{g}}:=\mathfrak{k}\oplus i\mathfrak{p}$ is a Lie algebra and is a real form of $\mathfrak{g}^{\mathbb{C}}$. Moreover the complex metric $g^{\mathbb{C}}(-,-)$ restricts to an inner product on $\tilde{\mathfrak{g}}$ by construction. Thus if we let $\tilde{G}$ be the unique connected Lie subgroup of $G^{\mathbb{C}}$ (the real Lie group) with Lie algebra $\tilde{\mathfrak{g}}$, then the corollary follows. \end{proof}

In view of Remark \ref{signature} with the signature change $g\mapsto -g$, if $(G,g)$ can be Wick-rotated to a signature $(-,-,\cdots,-)$, then $(G,-g)$ can be Wick-rotated to a Riemannian space, thus there would exist a Cartan involution of $\mathfrak{g}$ w.r.t $-g$. We note in the Corollary that w.r.t the existing Cartan involution, then the Wick-rotated Riemannian Lie group may be chosen to be a virtual real form. Moreover note that since a Wick-rotation is a local condition then on Lie algebra level we have proved:

\begin{cor} \label{compatible} Let $(\mathfrak{g}^{\mathbb{C}}, g^{\mathbb{C}})$ be a holomorphic inner product space, where $\mathfrak{g}^{\mathbb{C}}$ is a complex Lie algebra. Let $\mathfrak{g}\subset \mathfrak{g}^{\mathbb{C}}$ be a real form which is a real slice. Assume there exist a compact real form $\mathfrak{u}\subset \mathfrak{g}^{\mathbb{C}}$ which is also a real Lie subalgebra. Let $\sigma$ be the conjugation map of $\mathfrak{g}$. Then there exist an automorphism $\phi\in Aut(\mathfrak{g}^{\mathbb{C}})\cap O(n,\mathbb{C})$ such that:
$\sigma(\phi(\mathfrak{u}))\subset \phi(\mathfrak{u}).$
  \end{cor}

Note in the corollary that if $\tau$ denotes the conjugation map of the compact real form $\phi(\mathfrak{u})\subset (\mathfrak{g}^{\mathbb{C}},g^{\mathbb{C}})$, then the map $\theta^{\mathbb{C}}:=\sigma\tau$ restricts to a Cartan involution $\theta$ of $\mathfrak{g}$.

Thus we have proved a general version of $\acute{E}$. Cartan's result: (\cite{Helgason}, Thm 7.1). Note also that the proof given there for the semi-simple case w.r.t to the Killing form is not valid for a general pair: $(\mathfrak{g},g)$ as above, indeed following the notation of the proof, it is not obvious that $N:=\sigma\tau\in O(n,\mathbb{C})\cap Aut(\mathfrak{g}^{\mathbb{C}})$.  

One shall note that it may be the case that a pseudo-Riemannian Lie group $(G,g)$ can be Wick-rotated to more than one Riemannian Lie group, in such a case we have the following (again one should compare this to semi-simple compact real forms w.r.t $-\kappa$):

\begin{prop}\label{local} Suppose there exist two Riemannian Wick-rotatable Lie groups: $(G,g)$ and $(\tilde{G},\tilde{g})$. Then $(G,g)$ and $(\tilde{G},\tilde{g})$ are locally isometric Lie groups. In particular if moreover $G$ and $\tilde{G}$ are both simply connected then $G$ and $\tilde{G}$ are isometric Lie groups.\end{prop}
\begin{proof} Choose a map $g\in O(n,\mathbb{C})$ mapping $\mathfrak{g}\mapsto \tilde{\mathfrak{g}}$. Consider the action and notation of Section \ref{bilaction}. Using the map $g$ the Lie bracket $v:=[-,-]$ of $\mathfrak{g}^{\mathbb{C}}$ lies in $\mathcal{V}$, and also $\tilde{v}:=g^{-1}\cdot v\in \mathcal{V}$. Thus $O(n,\mathbb{C})v\ni \tilde{v}$. But since $O(n)$ (the isometries of $(\mathfrak{g},g)$) is compact, then we may choose $h\in O(n)$ such that the vectors $v$ and $\tilde{v}$ lie in the same $O(n)$-orbit, i.e $h\cdot v=\tilde{v}$. Or in other words: $$gh\cdot v=v.$$ Now since $h$ maps $\mathfrak{g}$ to $\mathfrak{g}$ by definition and $gh$ fixes $v$, i.e fixes the complex Lie bracket. Then $gh\in O(n,\mathbb{C})$ is an automorphism of complex Lie algebras, and it maps $\mathfrak{g}\mapsto \tilde{\mathfrak{g}}$. Therefore since the metrics are left-invariant we can conclude that $(G, g)$ and $(\tilde{G},\tilde{g})$ are locally isometric Lie groups as required. Finally if $G$ and $\tilde{G}$ are both simply connected then since any local isometry is also an isometry, it follows that $(G,g)\cong (\tilde{G}, \tilde{g})$ are isometric Lie groups. The proposition is proved.  \end{proof}

Thus as a corollary for compact real forms:

\begin{cor} Let $(\mathfrak{g}^{\mathbb{C}}, g^{\mathbb{C}})$ be a holomorphic inner product space, where $\mathfrak{g}^{\mathbb{C}}$ is a complex Lie algebra. Let $\mathfrak{u}_1\subset \mathfrak{g}^{\mathbb{C}}\supset \mathfrak{u}_2$ be two real Lie subalgebras which are compact real forms. Then there exist a linear isomorphism: $\mathfrak{u}_1\xrightarrow{\phi} \mathfrak{u}_2$, such that $\phi\in O(n,\mathbb{C})\cap Aut(\mathfrak{g}^{\mathbb{C}})$.  \end{cor}

In the case of a complex semi-simple Lie group: $(G^{\mathbb{C}},-\kappa)$, equipped with the left-invariant Killing form, then any compact real form: $\mathfrak{u}\subset \mathfrak{g}$, gives rise to a real form: $U\subset G^{\mathbb{C}}$ (thus is by definition a Riemannian real slice of Lie groups). It follows by the theory of semi-simple Lie groups that any two compact real forms of $G^{\mathbb{C}}$ are isomorphic Lie groups, and thus also isometric Lie groups.

Let $(G,g)\subset (G^{\mathbb{C}}, g^{\mathbb{C}})$ be a real slice of Lie groups. Recall again the action of Section \ref{bilaction}, and consider the Lie bracket $[-,-]$ of $\mathfrak{g}^{\mathbb{C}}$. Thus $[-,-]\in \mathcal{V}$ (the bilinear forms $\mathfrak{g}^2\rightarrow \mathfrak{g}$.) Suppose as usual that the signature of $g$ is $(p,q)$. Then from real GIT there are a finite number of real $O(p,q)$-orbits in the complex orbit: $O(n,\mathbb{C})\cdot[-,-]$, i.e $$O(n,\mathbb{C})\cdot [-,-]\cap \mathcal{V}=\cup^{m}_{i=1} O(p,q)v_i, $$ for some $m\geq 1$. We shall put an equivalence relation on the real slices of Lie groups of $(G^{\mathbb{C}}, g^{\mathbb{C}})$ by the relation of local isometry. Let $[(G,g)]$ denote an equivalence class, thus $[(G,g)]=[(\tilde{G}, \tilde{g})]\Leftrightarrow (G,g)\cong (\tilde{G},\tilde{g})$ (locally).

We can thus generalise Proposition \ref{local} in the following sense:

\begin{thm} Let $(G,g)\subset (G^{\mathbb{C}}, g^{\mathbb{C}})$ be a real slice of Lie groups, and $(p,q)$ be the signature of $g$. Let $O(n,\mathbb{C})\cdot [-,-]\cap \mathcal{V}=\cup^{m}_{i=1} O(p,q)v_i.$ Then there are exactly $m$ equivalence classes (up to a local isometry) of real slices of Lie groups in $(G^{\mathbb{C}}, g^{\mathbb{C}})$ with signature $(p,q)$. In particular if $m=1$ then all real slices of Lie groups in $(G^{\mathbb{C}}, g^{\mathbb{C}})$ of signature $(p,q)$ are locally isometric.   \end{thm}
\begin{proof} Suppose $(\tilde{G}, \tilde{g})$ is Wick-rotated to $(G,g)$ of the same signature. Let $h\in O(n,\mathbb{C})$ be such that $h(\mathfrak{g})=\tilde{\mathfrak{g}}$, then $\tilde{v}:=h^{-1}\cdot [-,-]\in \mathcal{V}$ is in the same complex orbit as $[-,-]$. Thus we have a mapping of an equivalence class: $$[(\tilde{G},\tilde{g})]\mapsto O(p,q)\tilde{v}.$$ The map does not depend on the choice of $h$, since if $h_1\in O(n,\mathbb{C})$ also maps $h_1(\mathfrak{g})=\tilde{\mathfrak{g}}$, then $h^{-1}h_1\in O(p,q)$, and $h^{-1}h_1\cdot \tilde{v}=h^{-1}\cdot [-,-]$. The map is well-defined. Indeed let $(G_1, g_1)$ map to $O(p,q)v_1:=O(p,q)\cdot(h_1^{-1}\cdot [-,-])$ for some $h_1\in O(n,\mathbb{C})$ with $h_1(\mathfrak{g})=\mathfrak{g}_1$. Assume $(G_1,g_1)$ is locally isometric to $(\tilde{G},\tilde{g})$. Then there exist $g\in O(n,\mathbb{C})\cap Aut(\mathfrak{g}^{\mathbb{C}})$ such that $g(\mathfrak{g}_1)=\tilde{\mathfrak{g}}$, therefore: $$g_1:=h^{-1}gh_1\in O(p,q), \ \ g_1\cdot v_1=h^{-1}g\cdot [-,-]=h^{-1}\cdot [-,-]=\tilde{v},$$ using that $g$ fixes the Lie bracket.

To see that the map is injective, then suppose $[(G_j, g_j)]$ maps to the same orbit for $j=1,2$. Then by definition: $[(G_j, g_j)]\mapsto O(p,q)\cdot (h_j^{-1}\cdot [-,-])$ for maps $h_j\in O(n,\mathbb{C})$ with $h_j(\mathfrak{g})=\mathfrak{g}_j$. Thus since the orbits are the same, then choose $g\in O(p,q)$ such that $g\cdot (h_1^{-1}\cdot [-,-])=h_2^{-1}\cdot [-,-]$, i.e $h_2gh_1^{-1}\cdot [-,-]=[-,-]$ so that $h_2gh_1^{-1}\in O(n,\mathbb{C})\cap Aut(\mathfrak{g}^{\mathbb{C}})$. Note that $h_2gh_1^{-1}$ maps $\mathfrak{g}_1\mapsto \mathfrak{g}_2$. It follows that $[(G_1, g_1)]=[(G_2, g_2)]$ as required. 

It remains to show that the map is surjective. Indeed if $v_j\in \mathcal{V}$ is among the $v_1,\dots, v_m$, then there exist $h\in O(n,\mathbb{C})$ such that $h\cdot v_j=[-,-].$ If $V_1\subset \mathfrak{g}^{\mathbb{C}}$ denotes the real form (of vector spaces) $h(\mathfrak{g})$, then: $$[V_1,V_1]=h\Big{(}v(h^{-1}(V_1),h^{-1}(V_1))\Big{)}\subset h(v(\mathfrak{g},\mathfrak{g}))\subset h(\mathfrak{g}):=V_1.$$ Therefore $V_1$ is a real form of Lie algebras, thus redefine $V_1:=\mathfrak{g}_1$. Let $G_1$ be the virtual Lie subgroup of $G^{\mathbb{C}}$ with Lie algebra $\mathfrak{g}_1$, then $G_1$ is a real slice of Lie groups of signature $(p,q)$. Thus $[(G_1,g_1)]\mapsto O(p,q)v_j$, which proves that the map is surjective. The theorem is proved.
    \end{proof}

There are classes of Lie algebras with $m=1$, for instance the trivial case of abelian Lie algebras. However in general $m\neq 1$. Indeed even a semi-simple Lie algebra is not determined by the signature of its Killing form: $-\kappa$. As an example consider the semi-simple real forms $\mathfrak{o}(1,4)\subset (\mathfrak{o}(5,\mathbb{C}), -\kappa)\supset \mathfrak{o}(2,3)$. Then the signatures are $(6,4)$ and $(4,6)$ respectively. Thus $\mathfrak{o}(1,4)\oplus \mathfrak{o}(2,3)$ is a real form of $(\mathfrak{o}(5,\mathbb{C})^2, -\kappa)$ of signature $(10,10)$. But also if $\mathfrak{o}(5,\mathbb{C})_{\mathbb{R}}$ denotes the real Lie algebra of $\mathfrak{o}(5,\mathbb{C})$, then it is also a real form of $\mathfrak{o}(5,\mathbb{C})^2$ which is simple, also of signature $(10,10)$, thus $$\mathfrak{o}(5,\mathbb{C})_{\mathbb{R}}\ncong \mathfrak{o}(1,4)\oplus \mathfrak{o}(2,3),$$ and so $m\geq 2$ in this example. 

We now give two examples, one where a Lie group is Wick-rotatable to a Riemannian Lie group, and the other where a Lie group is not Wick-rotatable to a Riemannian Lie group.

\begin{ex} \label{heis} Let $H_3(\mathbb{R})\subset H_3(\mathbb{C})$ be the 3-dimensional real and complex Heisenberg groups. The Lie algebra of $H_3(\mathbb{R})$ denoted: $\mathfrak{h}_3(\mathbb{R})$, is the set of strictly upper triangular $3\times 3$ matrices. A basis of the Lie algebra is given by $\{e_1,e_2,e_3\}$ with, $$[e_1,e_2]=e_3, \ \ [e_1,e_3]=[e_2,e_3]=0.$$ We may identify $\{e_j\}_j$ with the standard basis of $\mathbb{R}^3$. Let $g(-,-)$ be the standard Lorentzian pseudo-inner product on $\mathbb{R}^3$, i.e of signature $(+,+,-)$. Thus $(H_3(\mathbb{R}), -g)$ is a real slice (of Lie groups) of $(H_3(\mathbb{C}), -g^{\mathbb{C}})$. Note that $g(-,-)$ is not bi-invariant, since $g([e_1,e_2], e_3)=-1\neq g(e_1, [e_2,e_3])=0.$ Define the linear map: $\theta\in End(\mathfrak{h}_3(\mathbb{R}))$ by: $$\lambda_1 e_1+\lambda _2 e_2+\lambda_3 e_3\mapsto -\lambda_1 e_1-\lambda _2 e_2+\lambda_3 e_3.$$ Then it is easy to show that this is an involution of Lie algebras, and moreover $\theta$ is a Cartan involution of $\mathfrak{h}_3(\mathbb{R})$ w.r.t $-g(-,-)$, thus by Corollary \ref{Riem} it follows that $H_3(\mathbb{R})$ can be Wick-rotated to a Riemannian Lie group $\tilde{G}$. Note that $\tilde{G}$ is the real form of $H_3(\mathbb{C})$ consisting of matrices of the form: $\begin{bmatrix} 1 & ix & iy \\ 0 & 1 & z \\ 0 & 0 & 1 \\ \end{bmatrix}$ for $x,y,z\in \mathbb{R}$.   \end{ex}

\begin{ex} \label{ex2} Consider the real form: $G:=SL_2(\mathbb{R})^2\subset G^{\mathbb{C}}:=SL_2(\mathbb{C})^2$. Then we can equip $G$ with a left-invariant metric $g(-,-)$ of signature $(3,3)$, by equipping one copy with $-\kappa$ and the other copy with $\kappa$. The real forms up to isomorphism of $\mathfrak{sl}_2(\mathbb{C})^2$ are: $$\mathfrak{sl}_2(\mathbb{R})^2, \ \mathfrak{sl}_2(\mathbb{R})\oplus \mathfrak{su}(2), \ \mathfrak{su}(2)^2, \ \mathfrak{o}(1,3).$$ Let $\tilde{\mathfrak{g}}$ be one of these real forms (except the last one), then we may Wick-rotate $G$ to the corresponding real forms $\tilde{G}$ of $SL_2(\mathbb{C})^2$ of signature either: $(3,3)$ or $(1,5)$. In the case of Wick-rotating to $(SU(2)^2, \tilde{g})$ we get a signature of $(3,3)$. Now note that if $G$ can be Wick-rotated to a signature: $(0,6)$ or Riemannian: $(6,0)$, then we can find (by Corollary \ref{Riem}) a Cartan involution of $\mathfrak{su}(2)^2$ w.r.t $-\tilde{g}$ or $+\tilde{g}$ respectively: $$\mathfrak{su}(2)^2\xrightarrow{\theta} \mathfrak{su}(2)^2.$$ Suppose the Cartan involution is w.r.t $\tilde{g}$. Then if $\mathfrak{su}(2)^2=\mathfrak{k}\oplus \mathfrak{p},$ is the Cartan decomposition w.r.t $\theta$, we have $\mathfrak{g}_1:=\mathfrak{k}\oplus i\mathfrak{p}\cong \mathfrak{o}(1,3)$. Indeed $\theta^{\mathbb{C}}$ is a Cartan involution of $\mathfrak{g}_1$ (w.r.t $-\kappa$), thus $-\kappa$ has signature $(3,3)$, hence it must be the case that $\mathfrak{g}_1\cong \mathfrak{o}(1,3)$. 

Now consider the copy $\mathfrak{o}(1,3)$ identified as the real form $$\{(x,\tau(x))|x\in \mathfrak{sl}_2(\mathbb{C})\}\subset \mathfrak{sl}_2(\mathbb{C})^2.$$ We can extend the inner product $g^{\mathbb{C}}:=g_0$ on $\mathfrak{g}_1$ to $\mathfrak{o}(1,3)$ by using an isomorphism $\phi$ (of Lie algebras) $\mathfrak{g}_1\cong \mathfrak{o}(1,3)$, thus by complexifying we get a holomorphic inner product $b$ on $\mathfrak{sl}_2(\mathbb{C})^2$. On each copy of $\mathfrak{sl}_2(\mathbb{C})$ we get $b=\lambda\kappa$ where $\kappa$ is the Killing form on $\mathfrak{sl}_2(\mathbb{C})$, thus we may assume the holomorphic inner product is of the form $b=\lambda_1 \kappa+\lambda_2\kappa$. Using that $g^\mathbb{C}=-\kappa\oplus \kappa$, then necessarily $\lambda_1, \lambda_2$ are real. Now for $X:=(x,\tau(x))\in \mathfrak{o}(1,3)$ we get \begin{align*}b(X,X)&=\lambda_1\kappa(x,x)+\lambda_2\kappa(\tau(x),\tau(x))\\&=\lambda_1\kappa(x,x)+\lambda_2\overline{\kappa(x,x)}\\&=(\lambda_1+\lambda_2)Re(\kappa(x,x))+i(\lambda_1-\lambda_2)Im(\kappa(x,x)). \end{align*}  Thus since $b(X,X)$ is real, then necessarily $\lambda_1=\lambda_2$, so $b$ is proportional to the Killing form on $\mathfrak{sl}_2(\mathbb{C})^2$, and so since $\phi^{\mathbb{C}}$ is an isomorphism: $$\Big{(}\mathfrak{sl}_2(\mathbb{C})^2, g^{\mathbb{C}}\Big{)}\cong \Big{(}\mathfrak{sl}_2(\mathbb{C})^2, b\Big{)},$$ then we deduce that $g^{\mathbb{C}}=-\kappa\oplus \kappa$ is also proportional to the Killing form, this is a contradiction. The argument for the signature case: $(0,6)$, is identical with the change: $g\mapsto -g$. We conclude that $(G,g)$ can not be Wick-rotated to a Riemannian Lie group nor to a Lie group of signature $(0,6)$. 

   \end{ex}

One shall note that Proposition \ref{local} does not hold for a general non-Riemannian signature. Indeed consider the previous example then $SL_2(\mathbb{R})^2$ has signature $(3,3)$ and can be Wick-rotated to $SU(2)^2$ also of signature $(3,3)$, but they are not locally isometric (since their Lie algebras are non-isomorphic). Thus $m\geq 2$ in the previous theorem. 

We end this section by considering a result on semi-simple Lie groups.                                  

\begin{prop} Let $(G,g)\subset (G^{\mathbb{C}},g^{\mathbb{C}})$ be a real slice, and $G$ be semi-simple. Then $(G,g)$ can be Wick-rotated to a Riemannian compact Lie group if and only if there exist a Cartan involution $\theta$ of $\mathfrak{g}$ (w.r.t $g$) which is also a Cartan involution of $\mathfrak{g}$ (w.r.t $-\kappa$). \end{prop}
\begin{proof} $(\Rightarrow)$. If $(G,g)$ is Wick-rotated to a Riemannian Lie group, then by Corollary \ref{Riem}, we can choose a Cartan involution $\theta$ of $\mathfrak{g}$. Denote $\mathfrak{g}=\mathfrak{k}\oplus \mathfrak{p}$ for the Cartan decomposition. Then following the proof of Corollary \ref{Riem}, then we can find a Riemannian Lie group $\tilde{G}$ with Lie algebra: $\tilde{\mathfrak{g}}:=\mathfrak{k}\oplus i\mathfrak{p}$, which is Wick-rotated to $G$ in $(G^{\mathbb{C}},g^{\mathbb{C}})$. By Proposition \ref{local}, $\tilde{\mathfrak{g}}$ is compact, since we can Wick-rotate $G$ to a Riemannian compact Lie group (by assumption). But since $\theta$ is a Cartan involution of $\mathfrak{g}$ (w.r.t $-\kappa$) if and only if $\tilde{\mathfrak{g}}$ is compact, then the direction is proved. $(\Leftarrow)$. Suppose $\theta$ is a Cartan involution of $\mathfrak{g}$ w.r.t $g(-,-)$ and $-\kappa$ simultaneously. Thus if $\mathfrak{u}:=\mathfrak{k}\oplus i\mathfrak{p}$ is the compact real form of $\mathfrak{g}^{\mathbb{C}}$ associated with $\theta$, then there exist a compact real form $U\subset G^{\mathbb{C}}$ with Lie algebra $\mathfrak{u}$. The proposition follows.    \end{proof}

Thus since we may lift a local Cartan involution: $\mathfrak{g}\xrightarrow{\theta}\mathfrak{g}$, to a global Cartan involution: $G\xrightarrow{\Theta} G$, then in view of the previous proposition, there is a $\Theta$ which is an isometry of $(G,g)$, i.e $\Theta\in Isom(G)$. Observe also that if there exist a real slice of Lie groups of $G^{\mathbb{C}}$ which is compact Riemannian, then the possible signatures $(p,q)$ w.r.t $g^{\mathbb{C}}$ is a subset of the possible signatures of $-\kappa$ (of $\mathfrak{g}^{\mathbb{C}}$).

It is tempting to think that if $(G,g)$ can be Wick-rotated to a compact semi-simple Riemannian Lie group, then it would be locally isometric to $(G,-\lambda\kappa)$ ($\lambda>0$). However this is false, indeed consider $G:=SL_2(\mathbb{R})^2$ equipped with the metric $g:=-\kappa\oplus -2\kappa$. We can Wick-rotate to the compact Riemannian Lie group: $SU(2)^2$. Consider the real form $\cong \mathfrak{o}(1,3)$ identified with the set: $$\{(x,\tau(x))|x\in \mathfrak{sl}_2(\mathbb{C})_{\mathbb{R}}\}\subset \mathfrak{sl}_2(\mathbb{C})^2,$$ where $\tau$ is the conjugation map of $\mathfrak{sl}_2(\mathbb{C})$ with fix-points $\mathfrak{su}(2)$. If $(G,g)$ is locally isometric to $\lambda\kappa$ for some $\lambda\in \mathbb{R}$, then we can Wick-rotate to a $\tilde{G}\subset G^{\mathbb{C}}$ with Lie algebra $\mathfrak{o}(1,3)$. However $\mathfrak{o}(1,3)$ on $g^{\mathbb{C}}$ is not a real slice. We thus conclude that $(G,g)$ is not locally isometric to $(G,\lambda\kappa)$ for any $\lambda\in \mathbb{R}$.

\begin{rem} One shall observe that if $(G,g)$ and $(\tilde{G},\tilde{g})$ are pseudo-Riemannian spaces, where $G$ and $\tilde{G}$ are Lie groups, but the metrics are not assumed to be left-invariant, then the proof of Theorem \ref{t} is still valid. The direction $(\Rightarrow)$ of Corollary \ref{Riem} is also valid, however the direction $(\Leftarrow)$ does not necessarily hold. \end{rem}

\section{Conjugacy of Cartan involutions}

Given a pseudo-Riemannian Lie group $(G,g)$, with two Cartan involutions $\theta_j$ ($j=1,2$) of $\mathfrak{g}$, one may wonder if they are conjugate in $Aut(\mathfrak{g})$. This is in fact true as we will show here, and we note again the resemblance with semi-simple Lie groups $G$ and Cartan involutions of $\mathfrak{g}$ (w.r.t $-\kappa$).

\begin{thm}\label{conjugate} Suppose $(G,g)\subset (G^{\mathbb{C}}, g^{\mathbb{C}})$ is a pseudo-Riemannian Lie group. Assume there exist two Cartan involutions: $\theta_1,\theta_2$ of $\mathfrak{g}$. Then $\theta_1$ is conjugate to $\theta_2$ in $Aut(\mathfrak{g})_0\cap O(p,q)_0$. \end{thm} 
\begin{proof} Write $\mathfrak{g}=\mathfrak{k}_1\oplus \mathfrak{p}_1=\mathfrak{k}_2\oplus \mathfrak{p}_2$ for the Cartan decompositions w.r.t $\theta_1$ and $\theta_2$ respectively. Denote also: $\mathfrak{k}_j\oplus i\mathfrak{p}_j:=\mathfrak{u}_j$ ($j=1,2$) for the real forms of $\mathfrak{g}^{\mathbb{C}}$. There exist Wick-rotations of $G$ to connected virtual real forms: $U_j\subset G^{\mathbb{C}}$ with Lie algebras $\mathfrak{u}_j$ which are Riemannian (by Corollary \ref{Riem}). If $\sigma$ denotes the conjugation map w.r.t $\mathfrak{g}$, and $\tau_j$ denotes the conjugation map of $\mathfrak{u}_j$, then we have $\theta^{\mathbb{C}}_j=\sigma\tau_j$. Now since $\theta_1$ is conjugate to $\theta_2$ in $O(p,q)_0$ (see Remark \ref{Ocon}), i.e there is a $\phi\in O(p,q)_0$ such that $$\phi\theta_1\phi^{-1}=\theta_2,$$ as linear maps, then $g:=\phi^{\mathbb{C}}$ sends $\mathfrak{u}_1\mapsto \mathfrak{u}_2$. Consider the action in Section \ref{bilaction} and the notation there. If $v:=[-,-]$ is the complex Lie bracket of $\mathfrak{g}^{\mathbb{C}}$, then $v\in \mathcal{V}$ and $w:=g^{-1}\cdot v\in \mathcal{W}$ (i.e is a bilinear map $\mathfrak{u}_1^2\rightarrow \mathfrak{u}_1$) lie in the same complex orbit. Note that $(\mathfrak{g},\mathfrak{u}_1)$ is a compatible pair. Now $w$ is a minimal vector since it belongs to $\mathcal{W}$, i.e: $$w\in O(n)_0w\cap O(p,q)_0v=K_0v,$$ where $$K:=\{g\in O(p,q)|g\theta_1=\theta_1 g\}\subset O(p,q),$$ is the maximal compact subgroup associated with the fixed global Cartan involution of $O(p,q)$: $g\mapsto \theta_1 g\theta_1$. Thus there is an element $k_0\in K_0\subset O(p,q)_0$ such that $k_0v=w$, in other words: $gk_0\cdot v=v$. Hence $gk_0\in O(p,q)_0\cap Aut(\mathfrak{g})$, and it follows that: $[\sigma, gk_0]=0$, i.e $$\theta_2=gk_0\circ\theta_1\circ k^{-1}_0g^{-1}.$$ The corollary is proved.   \end{proof}                                 

Thus on Lie algebras we get the following nice corollary:

\begin{cor} Let $(\mathfrak{g}^{\mathbb{C}}, g^{\mathbb{C}})$ be a holomorphic inner product space, where $\mathfrak{g}^{\mathbb{C}}$ is a complex Lie algebra. Let $\mathfrak{g}\subset \mathfrak{g}^{\mathbb{C}}$ be a real form which is a real slice. Then any two Cartan involutions of $\mathfrak{g}$ are conjugate in $Aut(\mathfrak{g})_0\cap O(p,q)_0$.  \end{cor}

Note that the corollary is a generalised version (for general pseudo-inner product spaces) of $\acute{E}$. Cartan's result: (\cite{Helgason}, Thm 7.2). Let us give an example (of a non-compact real form) where there is a unique Cartan involution of the Lie algebra:

\begin{ex} If we consider again the real Heisenberg group: $(H_3(\mathbb{R}),-g)$ and follow Example \ref{heis} with Cartan involution $\theta$, then calculating the derivation algebra of $\mathfrak{h}_3(\mathbb{R})$ w.r.t to the basis $\{e_j\}$, then the matrices have the form: $\begin{bmatrix} a & c & 0 \\ e & b & 0 \\ f & l & a+b \\ \end{bmatrix}$, for $a,b,c,e,l,f\in \mathbb{R}$. Thus $Dim\Big{(}\mathfrak{der}(\mathfrak{h}_3(\mathbb{R}))\Big{)}=6$. Now if such a derivation $D$ belongs to $\mathfrak{o}(1,2)$, then an easy calculation shows that $D=\begin{bmatrix} 0 & c & 0 \\ -c & 0 & 0 \\ 0 & 0 & 0 \\ \end{bmatrix}$, i.e the Lie algebra of $Aut(\mathfrak{h}_3(\mathbb{R}))\cap O(1,2)$ has dimension 1. Now a global Cartan involution $\Theta_1$ of $O(1,2)$ is given by $$f\mapsto \theta f\theta, \ \ O(1,2)=Ke^{\mathfrak{p}},$$ and clearly since $\theta\in Aut(\mathfrak{h}_3(\mathbb{R}))$, then it leaves $H:=Aut(\mathfrak{h}_3(\mathbb{R}))\cap O(1,2)$ invariant. But since the Lie algebra $\mathfrak{h}$ of $H$ is fixed by the corresponding local Cartan involution of $\mathfrak{o}(1,2)$, then it follows that $\Theta_1$  fixes pointwise all elements of $H$. Indeed this follows since $H$ is algebraic (see for example \cite{Neeb}), so every $h\in H$ can be written as $h=ke^x$ for $k\in K\cap H, \ x\in \mathfrak{h}\cap \mathfrak{p}$. Thus we conclude that: $$[\theta, f]=0, \forall f\in H,$$ in other words by the previous theorem, there exist a unique Cartan involution of $\mathfrak{h}_3(\mathbb{R})$, namely $\theta$.  \end{ex}

Recall that for a real semi-simple Lie algebra $\mathfrak{g}$ equipped with the Killing form: $-\kappa$. Then it is proved in Helgason (\cite{Helgason}) that given any involution $\tilde{\theta}$ of $\mathfrak{g}$ there exist a Cartan involution of $\mathfrak{g}$ commuting with $\tilde{\theta}$. We can also prove a generalised version of this result for a general pseudo inner product space: $(\mathfrak{g}, g)$, by mimicking the proof given for semi-simple Lie algebras in \cite{Helgason} together with Corollary \ref{compatible}.     

\begin{cor} \label{commuteC} Let $(\mathfrak{g}^{\mathbb{C}}, g^{\mathbb{C}})$ be a holomorphic inner product space, where $\mathfrak{g}^{\mathbb{C}}$ is a complex Lie algebra. Let $\mathfrak{g}\subset \mathfrak{g}^{\mathbb{C}}$ be a real form which is a real slice. Suppose there exist a compact real form of $\mathfrak{g}^{\mathbb{C}}$ which is also a real Lie subalgebra. Let $\tilde{\theta}\in Aut(g)\cap O(p,q)$ be an involution of $\mathfrak{g}$. Then there exist a Cartan involution $\theta$ of $\mathfrak{g}$ that commutes with $\tilde{\theta}$, i.e $[\tilde{\theta}, \theta]=0$. \end{cor}                           
\begin{proof} Let $\theta'$ be a Cartan involution of $\mathfrak{g}$ by Corollary \ref{compatible}. By mimicking the proof of (\cite{Helgason}, Thm 7.1) in view of Exercise (4, Ch.3, \cite{Helgason}), we apply the proof given there to the inner product $g_{\theta'}(-,-):=g(-,\theta'(-))$, together with the symmetric operator: $N:=\tilde{\theta}\theta'$. Thus there exist a $\psi\in O(p,q)\cap Aut(\mathfrak{g})$ such that $[\psi\theta'\psi^{-1},\tilde{\theta}]=0$, therefore let $\theta:=\psi\theta'\psi^{-1}$. \end{proof}                                  

Suppose now that $\Big{(}\mathfrak{g}^{\mathbb{C}}, g^{\mathbb{C}}\Big{)}$ is a holomorphic inner product space on a complex Lie algebra $\mathfrak{g}^{\mathbb{C}}$. Let $O(n,\mathbb{C})$ be the isometry group. If $\phi\in O(n,\mathbb{C})\cap Aut(\mathfrak{g}^{\mathbb{C}})$ is an involution such that when $\mathfrak{g}^{\mathbb{C}}=V_+\oplus V_-,$ is the eigenspace decomposition then $\mathfrak{g}:=V_+\oplus iV_-$ is a real slice (i.e $g^{\mathbb{C}}(\mathfrak{g},\mathfrak{g})\in \mathbb{R}$), then we shall write $\phi\in\mathcal{O}$ for such a map. We can put an equivalence relation on maps $\mathcal{O}$ by conjugacy in $O(n,\mathbb{C})\cap Aut(\mathfrak{g}^{\mathbb{C}})$. 

The following theorem should be compared with a similar result of semi-simple Lie algebras equipped with their Killing form (see for example \cite{OVi}, Thm 1.3):

\begin{thm} If $\Big{(}\mathfrak{g}^{\mathbb{C}}, g^{\mathbb{C}}\Big{)}$ have a compact real form, then there is a bijection between isomorphism classes of real forms $\mathfrak{g}\subset \Big{(}\mathfrak{g}^{\mathbb{C}}, g^{\mathbb{C}}\Big{)}$ and conjugacy classes of $\mathcal{O}$. \end{thm}
\begin{proof} Let $\mathfrak{g}\subset \Big{(}\mathfrak{g}^{\mathbb{C}}, g^{\mathbb{C}}\Big{)}$ be a real form, then we can choose a Cartan involution of $\mathfrak{g}$ say $\theta$ by (Corollary \ref{compatible}). Define the map $[\mathfrak{g}]\mapsto [\theta^{\mathbb{C}}].$ The map is well-defined (Theorem \ref{conjugate}) since any two Cartan involutions are conjugate in $O(p,q)\cap Aut(\mathfrak{g})$, where $(p,q)$ is the signature of the induced pseudo-inner product from $g^{\mathbb{C}}$. To see that the map is surjective, let $\phi\in \mathcal{O}$, and set $\mathfrak{g}:=V_+\oplus iV_-$ for the real form $\mathfrak{g}\subset \Big{(}\mathfrak{g}^{\mathbb{C}}, g^{\mathbb{C}}\Big{)}$. Then $\phi$ restricted to $\mathfrak{g}$ is an involution, and if $\sigma$ denotes its conjugation map then $[\sigma,\phi]=0$. But we may choose a Cartan involution $\theta$ of $\mathfrak{g}$ such that $[\theta,\phi]=0$ by (Corollary \ref{commuteC}). Thus $\theta^{\mathbb{C}}=\sigma\tau$ for $\tau$ a conjugation map of a compact real form $\mathfrak{u}\subset (\mathfrak{g}^{\mathbb{C}}, g^{\mathbb{C}})$. Thus $$\sigma\tau\phi=\phi\sigma\tau,$$ therefore $$\sigma\tau\phi=\sigma\phi\tau,$$ or in other words by canceling $\sigma$ we obtain $[\tau,\phi]=0$ so that $\phi$ is in fact a Cartan involution of $\mathfrak{g}$, and hence $[\mathfrak{g}]\mapsto [\phi]$. 

Suppose now that $\mathfrak{g}_j$ are two real forms for $j=1,2$ such that the images are the same: $[\theta_1^{\mathbb{C}}]=[\theta_2^{\mathbb{C}}]$. Then if $\sigma_j$ denotes the conjugation maps, and $\mathfrak{u}_j$ are the compact real forms compatible with $\mathfrak{g}_j$, then $\theta_j^{\mathbb{C}}=\sigma_j\tau_j$. But since the maps are conjugate in $O(n,\mathbb{C})\cap Aut(\mathfrak{g}^{\mathbb{C}})$, say by $\phi$, thus $\phi\theta_1^{\mathbb{C}}\phi^{-1}=\theta_2^{\mathbb{C}}$ then it is easy to see that: $\phi(\mathfrak{u}_1)=\mathfrak{u}_2$. Thus $$\theta_2^{\mathbb{C}}=\phi\theta_1^{\mathbb{C}}\phi^{-1}=\phi\sigma_1\tau_1\phi^{-1}=\phi\sigma_1\phi^{-1}\phi\tau_1\phi^{-1}=\phi\sigma_1\phi^{-1}\tau_2,$$ thus cancelling $\tau_2$ we obtain: $\phi\sigma_1\phi^{-1}=\sigma_2$, which proves that $[\mathfrak{g}_1]=[\mathfrak{g}_2]$, and hence the map is injective. The theorem is proved. \end{proof}

\section{Wick-rotating a Lorentzian signature}

If we assume our left-invariant metric on our Lie group $G$ is Lorentzian or of signature $(+,-,\cdots,-)$, then being able to Wick-rotate to a Riemannian space puts some constraints on the structure of the Lie algebra (in view of Corollary \ref{Riem}). Now since a Wick-rotation is a local condition, it would be interesting to know what type of Lie algebra allows for a Wick-rotation to a Riemannian Lie group. 
\\

We recall by the fundamental \textsl{Levi-Malcev theorem} (see for example \cite{OVi}) that our Lie algebra $\mathfrak{g}$ can be written as a semi-direct sum $\mathfrak{g}=\mathfrak{s}\ltimes \mathfrak{h}$, where $\mathfrak{h}$ is the radical of $\mathfrak{g}$ and $\mathfrak{s}\subset \mathfrak{g}$ is either trivial or a semi-simple subalgebra of $\mathfrak{g}$ called the Levi-factor. 

It is clear that $\mathfrak{g}^{\mathbb{C}}=\mathfrak{s}^{\mathbb{C}}\ltimes\mathfrak{h}^{\mathbb{C}}$, and if $\tilde{\mathfrak{g}}$ is another real form of $\mathfrak{g}^{\mathbb{C}}$, then writing a Levi-decomposition: $\tilde{\mathfrak{g}}=\tilde{\mathfrak{s}}\ltimes \tilde{\mathfrak{h}}$, then $\tilde{\mathfrak{h}}$ is a real form of $\mathfrak{h}^{\mathbb{C}}$. To see that $\tilde{\mathfrak{s}}$ is a real form of $\mathfrak{s}^{\mathbb{C}}$, we note that there exist a $k\geq 1$ such that $$\mathfrak{s}^{\mathbb{C}}=[{\mathfrak{g}^{\mathbb{C}}}^{(k)},{\mathfrak{g}^{\mathbb{C}}}^{(k)}]\supset [\tilde{\mathfrak{g}}^{(k)},\tilde{\mathfrak{g}}^{(k)}]=\tilde{\mathfrak{s}}.$$

In view of the existence of an involution of Lorentzian decomposition we can say the following:

\begin{prop} Let $(G,g)\subset (G^{\mathbb{C}},g^{\mathbb{C}})$ be a real slice of Lie groups. Then the following statements hold:
\begin{enumerate}
\item{} Suppose $g(-,-)$ has Lorentzian signature. If $(G,g)$ can be Wick-rotated to a Riemannian Lie group $(\tilde{G},\tilde{g})$ then $\mathfrak{s}=0$ or $\mathfrak{h}\neq 0$. Moreover if $\tilde{\mathfrak{s}}$ is a Levi-factor of $\tilde{\mathfrak{g}}$, then $\tilde{\mathfrak{s}}\cong \mathfrak{s}$.
\item{} Suppose $g(-,-)$ has signature $(+,-,\cdots,-)$. If $(G,g)$ can be Wick-rotated to a Riemannian Lie group, then either $\mathfrak{s}=0$ or $\mathfrak{s}\cong \mathfrak{sl}_2(\mathbb{R})$.
\end{enumerate}
 \end{prop} 
\begin{proof} For case (1), assume $\mathfrak{g}=\mathfrak{s}\ltimes \mathfrak{h}$ for $\mathfrak{s}\neq 0$, and choose a Cartan involution $\theta$ of $\mathfrak{g}$. Write $\mathfrak{g}=\mathfrak{k}\oplus \mathfrak{p}$ for the Cartan decomposition. Then $\theta$ leaves $\mathfrak{s}$ invariant: $\theta(\mathfrak{s})\subset \mathfrak{s}$. Indeed note that since $\mathfrak{h}$ is solvable, then there exist $k\geq 1$ such that the $k^{th}$-derived algebra satisfies: $[\mathfrak{g}^{(k)},\mathfrak{g}^{(k)}]=\mathfrak{s},$ thus it follows that $\theta$ must leave $\mathfrak{s}$ invariant, and hence we can write, $$\mathfrak{s}=(\mathfrak{s}\cap \mathfrak{k})\oplus (\mathfrak{s}\cap \mathfrak{p}).$$ We claim that $\mathfrak{s}\cap \mathfrak{p}=0$, indeed suppose not, i.e $\mathfrak{p}\subset \mathfrak{s}$ thus $[\mathfrak{s},\mathfrak{p}]\subset \mathfrak{p}$ so $\mathfrak{p}$ is an abelian non-trivial ideal of $\mathfrak{s}$, contradicting the semi-simplicity of $\mathfrak{s}$. Thus $\theta$ fixes $\mathfrak{s}$ point wise. Moreover $\mathfrak{p}\lhd \mathfrak{g}$ is an abelian ideal, and so therefore $\mathfrak{p}\subset \mathfrak{h}$, i.e $\mathfrak{h}\neq 0$. Finally since $g(-,-)$ restricted to $\mathfrak{s}$ and $\tilde{g}(-,-)$ restricted to $\tilde{\mathfrak{s}}$ is positive definite, then $\mathfrak{s}$ and $\tilde{\mathfrak{s}}$ give rise to a Wick-rotation of two Riemannian Lie groups, thus by Proposition \ref{local} it follows that $\mathfrak{s}\cong \tilde{\mathfrak{s}}$, and case (1) is proved. For case (2) suppose $\mathfrak{g}$ is non-solvable (i.e $\mathfrak{s}\neq 0$), then again w.r.t $\theta$ we see that $$\mathfrak{s}=(\mathfrak{s}\cap \mathfrak{k})\oplus (\mathfrak{s}\cap \mathfrak{p}),$$ where $\mathfrak{s}\cap \mathfrak{k}\neq 0$, since if not then $\mathfrak{s}\subset \mathfrak{p}$, i.e $\mathfrak{s}=[\mathfrak{s},\mathfrak{s}]\subset \mathfrak{k}$, which is a contradiction. Now since $\theta^{\mathbb{C}}$ is a Cartan involution of a real form $\tilde{\mathfrak{g}}\subset\mathfrak{s}^{\mathbb{C}}$, then $-\kappa$ on $\tilde{\mathfrak{g}}$ must also have the signature $(+,-,\cdots,-)$, this follows since $\mathfrak{k}^{\mathbb{C}}$ is 1-dimensional. Now finally if $\tilde{\mathfrak{g}}=\tilde{\mathfrak{k}}\oplus \tilde{\mathfrak{p}}$ is a Cartan decomposition, then $\tilde{\mathfrak{k}}$ is abelian and 1-dimensional, thus it follows that $\tilde{\mathfrak{g}}\cong \mathfrak{sl}_2(\mathbb{R})$ see for example (Prop. 13.1.10, \cite{Neeb}). We conclude that $\mathfrak{s}^{\mathbb{C}}\cong \mathfrak{sl}_2(\mathbb{C})$, and hence also $\mathfrak{s}\cong \mathfrak{sl}_2(\mathbb{R})$. The proposition is proved.     \end{proof}

Thus restricting to the class of semi-simple Lie algebras it is impossible to Wick-rotate a Lorentzian metric to a Riemannian metric. However even for a nilpotent Lie algebra the converse of (1) is not necessarily true, indeed consider the nilpotent Lie algebra $\mathfrak{h}_3(\mathbb{R})$ of $3\times 3$ strictly upper triangular matrices. Then if $\theta$ is an involution with $Dim(\mathfrak{p})=1$, we must be able to find a basis $\{x_1,x_2,x_3\}$ such that $$[x_1,x_2]=C^1_{12}x_1+C^2_{12}x_2, \ \ [x_1,x_3]=C^3_{13}x_3, \ \ [x_2,x_3]=C^3_{23}x_3.$$ But since $[\mathfrak{h}_3(\mathbb{R}),\mathfrak{h}_3(\mathbb{R})]$ has dimension 1, then it follows that $C^1_{12}=0=C^2_{12}$. Moreover since $\mathfrak{h}_3(\mathbb{R})$ is nilpotent of class 2, then we conclude also that $C^3_{12}=0=C^3_{23}$, i.e $\mathfrak{h}_3(\mathbb{R})$ would have to be abelian, thus it is not possible to Wick-rotate a Lorentzian metric (i.e of signature $(+,+,-)$) on $H_3(\mathbb{R})$ to a Riemannian metric. However Example \ref{heis}, shows that $H_3(\mathbb{R})$ may possess a metric of signature $(-,-,+)$ where this is possible.

In view of case (2), there are examples of metrics (of signature $(+,-,\cdots,-)$) that are Wick-rotatable to a Riemannian metric within: $\mathfrak{g}=\mathfrak{sl}_2(\mathbb{R})$ or $\mathfrak{g}=\mathfrak{h}_3(\mathbb{R})$ and even $\mathfrak{g}=\mathfrak{sl}_2(\mathbb{R})\oplus \mathfrak{h}_3(\mathbb{R})$.

If we impose the condition that the metric is bi-invariant, i.e $(\mathfrak{g},g)$ is a quadratic Lie algebra, then we have the following equivalence result:

\begin{cor} Let $(G,g)\subset (G^{\mathbb{C}},g^{\mathbb{C}})$ be a real slice of Lie groups. Suppose $g$ is bi-invariant. Then the following statements hold:
\begin{enumerate}
\item{} Suppose $g(-,-)$ has Lorentzian signature. Then $(G,g)$ can be Wick-rotated to a Riemannian Lie group $(\tilde{G},\tilde{g})$ if and only if $\mathfrak{g}$ is abelian or $\mathfrak{g}$ is a direct sum of $\mathfrak{h}\neq 0$ and $\mathfrak{s}$ is compact semi-simple. Moreover $\tilde{\mathfrak{g}}\cong \mathfrak{g}$.
\item{} Suppose $g(-,-)$ has signature $(+,-,\cdots,-)$. Then $(G,g)$ can be Wick-rotated to a Riemannian Lie group, if and only if either $\mathfrak{g}$ is abelian or $\mathfrak{g}$ is a direct sum of $\mathfrak{s}\cong \mathfrak{sl}_2(\mathbb{R})$ and $\mathfrak{h}$ abelian.
\end{enumerate}
 \end{cor}
\begin{proof} Case (1). Since $g$ is bi-invariant then $\tilde{\mathfrak{g}}$ must be reductive, because $\tilde{g}$ is bi-invariant and a Riemannian metric. Thus the complexification is also reductive, i.e so are the real forms, thus $\mathfrak{g}$ is reductive. This means that either $\mathfrak{s}=0$ or $\mathfrak{s}$ is semi-simple, and $\mathfrak{h}$ is abelian. So if $\mathfrak{g}$ is non-abelian then $\mathfrak{s}$ is semi-simple. Now by the proof of the previous proposition case (1), then given a Cartan involution of $\mathfrak{g}$ we must have that $\theta$ fixes point wise $\mathfrak{s}$. This means that $g$ restricted to $\mathfrak{s}$ is an inner product. If $\mathfrak{s}^{\mathbb{C}}$ is simple, then $g_{|_{\mathfrak{s}}}$ must be proportional to the Killing form: $\lambda\kappa$ ($\lambda\in \mathbb{R}$). Now since $g$ is positive definite on $\mathfrak{s}$ then $\lambda>0$ i.e $\mathfrak{s}$ is compact. If $\mathfrak{s}^{\mathbb{C}}$ is not simple then on each simple ideal, $g^{\mathbb{C}}$ is proportional to the Killing form. There are two cases to consider, either $\mathfrak{s}$ is simple (in which case $\mathfrak{s}$ has a complex structure) or each simple ideal $\mathfrak{J}$ of $\mathfrak{s}$ has a simple complexification $\mathfrak{J}^{\mathbb{C}}\lhd\mathfrak{s}^{\mathbb{C}}$ or has a complex structure. See for instance (Thm 6.94, \cite{Knapp}). Assume $\mathfrak{s}$ is simple, then $\mathfrak{s}^{\mathbb{C}}\cong \mathfrak{s}\oplus \mathfrak{s}$, where $\mathfrak{s}$ is a complex Lie algebra. Thus $g^{\mathbb{C}}$ restricted to $\mathfrak{s}$ is proportional to the complex Killing form on $\mathfrak{s}$, say $\lambda\kappa$. Thus viewing $\mathfrak{s}$ as a real Lie algebra, we get that $g^{\mathbb{C}}$ restricts to something proportional to the real part: $\lambda \frac{1}{2} Re(\kappa)=g$, which is positive definite by assumption. Therefore $\lambda\in\mathbb{R}$. But the real Killing form of $\mathfrak{s}$ is precisely $2Re(\kappa)$, so we conclude that either the Killing form is positive definite or negative definite, this is impossible. The argument for the other case is a combination of the previous two arguments. We conclude that $\mathfrak{s}$ is semi-simple compact. Now finally it follows that $\mathfrak{g}\cong \tilde{\mathfrak{g}}$ by the previous proposition and that $\mathfrak{h}\cong \tilde{\mathfrak{h}}$ (since they are abelian of the same dimension). 

Conversely if $\mathfrak{g}$ is abelian then the statement is trivial, therefore assume $\mathfrak{s}$ is compact semi-simple. Then $\mathfrak{s}:=[\mathfrak{g},\mathfrak{g}]$ and $\mathfrak{h}=\mathfrak{z}(\mathfrak{g})$ forms an orthogonal sum w.r.t $g$. Thus $g$ restricted to $\mathfrak{s}$ must be positive definite, indeed restricting $g$ on a compact simple ideal (which is a non-degenerate ideal) $\mathfrak{I}\lhd\mathfrak{s}$ we get something proportional to the Killing form on $\mathfrak{J}$: $\lambda\kappa$. Thus if $\lambda>0$ then this would contradict $g$ having Lorentzian signature. Therefore $\lambda<0$. Hence $g$ on $\mathfrak{h}$ must have Lorentzian signature, and so we can easily find a Cartan involution $\theta_{\mathfrak{h}}$ of $\mathfrak{h}$ such that $1_{\mathfrak{s}}\oplus \theta_{\mathfrak{h}}$ is a Cartan involution of $\mathfrak{g}$, now use Corollary \ref{Riem}. 

Case (2). Again since $\mathfrak{g}$ must be reductive, then by the previous proposition case (2), if $\mathfrak{g}$ is not abelian then $\mathfrak{s}\cong \mathfrak{sl}_2(\mathbb{R})$ and $\mathfrak{h}$ is abelian. Conversely let $\mathfrak{g}=\mathfrak{s}\oplus \mathfrak{h}$. If $\mathfrak{s}=0$, then the statements is obviously true. Suppose therefore that $\mathfrak{s}\cong \mathfrak{sl}_2(\mathbb{R})$. Note that $\mathfrak{s}=[\mathfrak{g},\mathfrak{g}]$ and $\mathfrak{h}=\mathfrak{z}(\mathfrak{g})$ is an orthogonal direct sum w.r.t $g$, i.e $[\mathfrak{g},\mathfrak{g}]^{\perp}=\mathfrak{z}(\mathfrak{g})$. Thus $g$ restricted to $\mathfrak{s}\cong\mathfrak{sl}_2(\mathbb{R})$ forms a quadratic Lie algebra, but since $\mathfrak{sl}_2(\mathbb{C})$ is simple, then $g$ must be proportional to the Killing form: $\lambda\kappa$ ($\lambda\in\mathbb{R}$). Note that $\lambda<0$ since other wise $g$ would not be able to have signature: $(+,-,-\dots, -)$. Also note that $g$ restricted to $\mathfrak{h}$ must be of signature: $(-,-,\dots,-)$. Thus choose any Cartan involution $\theta_{\mathfrak{s}}$ of $\mathfrak{s}$, and the Cartan involution $\theta_{\mathfrak{h}}$ of $\mathfrak{h}$ of the form: $\theta_{\mathfrak{h}}(x):=-x$. Then $\theta_{\mathfrak{s}}\oplus \theta_{\mathfrak{h}}$ is a Cartan involution of $\mathfrak{g}$, and the statement follows by Corollary \ref{Riem}.    \end{proof}

Thus a solvable Lie group $(G,g)$ with a bi-invariant (non-Riemannian) metric is not Wick-rotatable to a Riemannian Lie group.

\section{A remark on Wick-rotatable tensors of Lie groups}
Let $(G,g)\subset (G^{\mathbb{C}}, g^{\mathbb{C}})\supset (\tilde{G},\tilde{g})$ be a Wick-rotation of Lie groups. Assume $(\tilde{G}, \tilde{g})$ is Riemannian. By Corollary \ref{Riem} there exist a Cartan involution $\theta$ of $\mathfrak{g}$. Recall the section on Wick-rotatable tensors. We prove in this section that if $\tilde{v}=v\in \mathcal{V}\cap \tilde{\mathcal{V}}$ are two tensors  on the Lie algebras, then they are Wick-rotatable with respect to an embedding $\phi^{-1}\in H^{\mathbb{C}}$ into the same $H^{\mathbb{C}}$-orbit for $$H^{\mathbb{C}}:=Aut(\mathfrak{g}^{\mathbb{C}})\cap O(n,\mathbb{C})\subset O(n,\mathbb{C}),$$ such that $(\mathfrak{g},\phi^{-1}(\tilde{\mathfrak{g}}))$ is a compatible pair (i.e also a compatible triple). Denote $H:=Aut(\mathfrak{g})\cap O(p,q)$. Note that $H\subset H^{\mathbb{C}}$ is a real form. Indeed the real structure of $O(n,\mathbb{C})$ fixing $O(p,q)$ given by $A\mapsto\sigma A\sigma$ where $\sigma$ is the conjugation map w.r.t $\mathfrak{g}$, leaves $H^{\mathbb{C}}$ invariant, and thus fixes $H$. Note also that a global Cartan involution $\Theta$: $A\mapsto \theta A\theta$ of $O(p,q)$ where $\theta$ is a Cartan involution of $\mathfrak{g}$, also leave $H$ invariant. Thus $\Theta$ is a global Cartan involution of $H$. The arguments above are analogous for the real from: $\tilde{H}:=O(\tilde{p},\tilde{q})\cap Aut(\tilde{\mathfrak{g}})$. 

\begin{lem} \label{aut} Let $(G,g)\subset (G^{\mathbb{C}}, g^{\mathbb{C}})\supset (\tilde{G},\tilde{g})$ be a Wick-rotation of Lie groups. Assume $(\tilde{G}, \tilde{g})$ is Riemannian. Then any two tensors $v=\tilde{v}\in \mathcal{V}\cap\tilde{\mathcal{V}}$ can be embedded into the same $H^{\mathbb{C}}$-orbit for some $\phi^{-1}\in H^{\mathbb{C}}$ such that $\Big{(}\mathfrak{g},\phi^{-1}(\tilde{\mathfrak{g}})\Big{)}$ is a compatible pair.  \end{lem}
\begin{proof} By Corollary \ref{Riem}, we can choose a Cartan involution $\theta$ of $\mathfrak{g}$. Moreover by Proposition \ref{local} there is an isomorphism of Lie algebras: $\phi\in O(n,\mathbb{C})\cap Aut(\mathfrak{g}^{\mathbb{C}})$ sending $\mathfrak{u}\mapsto \tilde{\mathfrak{g}}$, where $\mathfrak{u}:=\mathfrak{k}\oplus i\mathfrak{p}$ w.r.t the Cartan decomposition of $\theta$. Thus $\phi^{-1}(\tilde{\mathfrak{g}})=\mathfrak{u}$ and $\mathfrak{g}$ are compatible. Let now $\tilde{v}=v\in \mathcal{V}\cap \tilde{\mathcal{V}}$ be two Wick-rotatable tensors, using the isometry tensor action of $\phi^{-1}$ on $v$, then $\phi^{-1}\cdot v$ and $v$ lie in the same $H^{\mathbb{C}}$-orbit as required. 
   \end{proof}

\begin{thm} \label{purely} Let $(G,g)\subset (G^{\mathbb{C}}, g^{\mathbb{C}})\supset (\tilde{G},\tilde{g})$ be a Wick-rotation of Lie groups. Assume $(\tilde{G}, \tilde{g})$ is Riemannian. Let $\tilde{v}=v\in \mathcal{V}\cap\tilde{\mathcal{V}}$ be two tensors (i.e they are also Wick-rotatable). There exist a Cartan involution $\theta$ of $\mathfrak{g}$ such that $\theta\cdot v=v$. \end{thm}
\begin{proof} Consider the real forms: $H\subset H^{\mathbb{C}}\supset \tilde{H}$ as above. Then one simply note that $H^{\mathbb{C}}$ is naturally algebraic, and moreover is a linearly complex reductive Lie group, simply because $\tilde{H}$ is a compact real form. Now since $O(n,\mathbb{C})$ and $Aut(\mathfrak{g}^{\mathbb{C}})$ are naturally algebraic groups defined over $\mathbb{R}$, then so is $H^{\mathbb{C}}$. Moreover $H$ and $\tilde{H}$ are the real points of $H^{\mathbb{C}}$ (respectively).  Thus the groups are naturally among the class of groups considered in Section \ref{compatiblereps}. The theorem follows by Lemma \ref{aut} and Theorem \ref{54}. \end{proof}

Thus we can restate a stronger version of Theorem \ref{a} for Lie groups, (see also paragraf after Definition \ref{RPM} for the tensors in question): 

\begin{thm}\label{b} Let $(G,g)\subset (G^{\mathbb{C}}, g^{\mathbb{C}})\supset (\tilde{G},\tilde{g})$ be a Wick-rotation of Lie groups. Assume $(\tilde{G}, \tilde{g})$ is a Riemannian Lie group. Then the following statements hold:
\begin{enumerate}
\item{} There exist a Cartan involution $\theta$ of $\mathfrak{g}$ such that 
 $\nabla_{\theta(x)}\theta(y)=\theta(\nabla_xy)$ for all $x,y\in \mathfrak{g}$.
\item{} There exist a Cartan involution $\theta$ of $\mathfrak{g}$ such that 
 $ric_g\Big{(}\theta(x),\theta(y)\Big{)}=ric_g(x,y)$ for all $x,y\in \mathfrak{g}$.
\item{} There exist a Cartan involution $\theta$ of $\mathfrak{g}$ such that 
$[\theta, Ric_g]=0$.
\item{} There exist a Cartan involution $\theta$ of $\mathfrak{g}$ such that 
$R\Big{(}\theta(x),\theta(y)\Big{)}\Big{(}\theta(z)\Big{)}=\theta\Big{(}R(x,y)(z)\Big{)}$ for all $x,y,z\in \mathfrak{g}$.
\end{enumerate}
 \end{thm} 

Any of the properties 1-4 of the previous theorem, are (using also Theorem \ref{t}) Wick-rotatable. Thus we state as a stronger result for Lie groups (compare with \cite{W1}):

\begin{cor} Let $(G,g)$ be a pseudo-Riemannian Lie group. Then the property of being Riemann purely electric (RPE) at $1$ w.r.t to a Cartan involution $\theta$ of $\mathfrak{g}$ is Wick-rotatable. \end{cor}
\begin{proof} Follows by Theorem \ref{t} and Theorem \ref{b}. \end{proof}

We end this section by also noting the following result:

\begin{cor} Let $(G,g)\subset (G^{\mathbb{C}}, g^{\mathbb{C}})\supset (\tilde{G},\tilde{g})$ be a Wick-rotation of Lie groups. Assume $(\tilde{G}, \tilde{g})$ is Riemannian. Then $$\Big{(}\forall x\in \mathfrak{g}\cap \tilde{\mathfrak{g}}\Big{)}\Big{(}\exists \theta\in Aut(\mathfrak{g})\Big{)}\Big{(}\theta(x)=x\Big{)},$$ where $\theta$ is a Cartan involution of $\mathfrak{g}$. \end{cor}
\begin{proof} Consider the isometry action of $O(n,\mathbb{C})$ (restricted to $H^{\mathbb{C}}$ defined above) on the complex Lie algebra: $\mathfrak{g}^{\mathbb{C}}$, i.e $$g\cdot x:=g(x), \ g\in O(n,\mathbb{C}), \ x\in \mathfrak{g}^{\mathbb{C}}.$$ Let $x=\tilde{x}\in \mathfrak{g}\cap\tilde{\mathfrak{g}}$. Then $x=\tilde{x}$ are two Wick-rotatable tensors, thus w.r.t a choice of $g\in H^{\mathbb{C}}$ we can assume that the real actions (of $H$ and $\tilde{H}$) are compatible. Moreover $x$ and $\tilde{x}$ lie in the same complex orbit, such that $O(p,q)x\sim O(n)\tilde{x}$ are compatible real orbits. We can now finish the proof by applying Theorem \ref{purely}. \end{proof}

\section{Wick-rotating an algebraic soliton}

A pseudo-Riemannian Lie group $(G,g)$, such that the Ricci operator $Ric_g\in \mathfrak{gl}(\mathfrak{g})$ has the form: $$Ric_g=\lambda\cdot 1_{\mathfrak{g}}+D,$$ where $\lambda\in \mathbb{R}$ and $D\in \mathfrak{der}(\mathfrak{g})$ (a derivation) is called an \textit{algebraic soliton} (defined in \cite{Asol}). If $D$ can be taken to be $D=0$, then $(G,g)$ is said to be Einstein, and moreover if the Lie algebra is also nilpotent (resp. solvable), then an algebraic soliton $(G,g)$, is said to be a \textit{Ricci nilsoliton} (resp. \textit{solsoliton}). For a discussion of Riemannian Ricci nilsolitons we refer to for example \cite{Lauret}. However we shall only be interested in Wick-rotating such a geometry.

We shall prove a result regarding the existence of a Wick-rotation of an algebraic soliton to a Riemannian Lie group, by using the results of the previous section.

\begin{lem} \label{Ric} The property of being an algebraic soliton is Wick-rotatable. \end{lem}
\begin{proof} Let $(G,g)$ be Wick-rotatable to $(\tilde{G}, \tilde{g})$. Suppose $(G,g)$ is an algebraic soliton. The Ricci operator: $\mathfrak{g}\xrightarrow{Ric_g} \mathfrak{g}$ on $G$ is also a restriction of the Ricci operator on $G^{\mathbb{C}}$. So if $Ric_g=\lambda\cdot 1_{\mathfrak{g}}+D$ for some $\lambda\in \mathbb{R}$ and $D\in \mathfrak{der}(\mathfrak{g})$, then also $$Ric_{g^{\mathbb{C}}}=(\lambda\cdot 1_{\mathfrak{g}})^{\mathbb{C}}+D^{\mathbb{C}}=\lambda 1_{\mathfrak{g}^{\mathbb{C}}}+D^{\mathbb{C}}.$$ Note that $D^{\mathbb{C}}$ is a derivation of $\mathfrak{g}^{\mathbb{C}}$, thus when restricting to $Ric_{\tilde{g}}$ we see that $D^{\mathbb{C}}$ must leave $\tilde{\mathfrak{g}}$ invariant and is thus a derivation $\tilde{D}$ of $\tilde{\mathfrak{g}}$ as required. The lemma follows. \end{proof}

Note from the lemma that $D\in End(\mathfrak{g})$ and $\tilde{D}\in End(\tilde{\mathfrak{g}})$ are Wick-rotatable tensors, under the isometry action: $$g\cdot f:=gfg^{-1}, \ g\in O(n,\mathbb{C}), \ f\in End(\mathfrak{g}^{\mathbb{C}}).$$

\begin{cor} The property of being a Ricci nilsoliton (resp. solsoliton) is Wick-rotatable.   \end{cor}

\begin{cor} The property of being Einstein is Wick-rotatable. \end{cor}

Applying the previous section, we get the following necessary condition for when an algebraic soliton can be Wick-rotated to a Riemannian algebraic soliton:

\begin{thm} \label{der}Suppose $(G,g)$ is an algebraic soliton, with $Ric_g=\lambda\cdot 1_{\mathfrak{g}}+D$, which can be Wick-rotated to a Riemannian algebraic soliton: $(\tilde{G},\tilde{g})$ with $Ric_{\tilde{g}}=\lambda\cdot 1_{\tilde{g}}+\tilde{D}$. Then there exist a Cartan involution $\theta$ of $\mathfrak{g}$ such that $[\theta, D]=0$. \end{thm}
\begin{proof} The derivations: $D^{\mathbb{C}}=\tilde{D}^{\mathbb{C}}\in \mathcal{V}\cap \tilde{\mathcal{V}}$ are Wick-rotatable (see the proof of Lemma \ref{Ric}). Thus w.r.t a choice of map $g\in H^{\mathbb{C}}$ we can assume w.l.o.g that $D$ and $\tilde{D}$ lie in the same complex orbit under the conjugation action: $H^{\mathbb{C}}\cdot D\ni\tilde{D}.$ By Theorem \ref{purely} there is a Cartan involution $\theta$ of $\mathfrak{g}$ such that $\theta\cdot D:=\theta D\theta=D$ or in other words $[\theta, D]=0$. The theorem is proved. \end{proof}

\begin{ex} We follow Example \ref{heis}. Thus consider the real 3-dimensional Heisenberg group $(H_3(\mathbb{R}), -g)$, then this is a Ricci nilsoliton of signature $(+,-,-)$. Indeed one can calculate with respect to the basis $\{e_1,e_2,e_3\}$, that the Ricci operator can be written uniquely as: $$Ric_{-g}=-\frac{3}{2}\cdot I_3+D,$$ where $D(e_1)=e_1, \ D(e_2)=e_2, \ D(e_3)=2e_3.$ Note that the Cartan involution $\theta$ commutes with $D$, i.e $[\theta, D]=0$. Thus when restricting to the Wick-rotated Riemannian Ricci nilsoliton $(\tilde{G}, \tilde{g})$ with Lie algebra: $\tilde{\mathfrak{g}}:=\langle ie_1, ie_2, e_3\rangle\subset \mathfrak{h}_3(\mathbb{C})$, we get the corresponding Ricci operator expressed as: $$Ric_{\tilde{g}}=-\frac{3}{2}\cdot I_3+\tilde{D},$$ with $\tilde{D}(ie_1)=ie_1, \ \tilde{D}(ie_2)=ie_2, \ \tilde{D}(e_3)=2e_3$. \end{ex}

\end{document}